\documentclass{amsart}
\usepackage[T1]{fontenc}
\usepackage{amsfonts}
\usepackage{amssymb}
\usepackage{amsmath}
\usepackage{amsthm}

\newtheorem{theorem}{Theorem}[section]
\newtheorem{lemma}[theorem]{Lemma}
\newtheorem{proposition}[theorem]{Proposition}

\theoremstyle{definition}

\newtheorem{remark}[theorem]{Remark}

\newcommand{\ud}{\mathrm{d}}
\numberwithin{equation}{section}
\begin{document}
\pagenumbering{arabic}

\title[Edge scaling limits for non-Hermitian random matrices]{Edge scaling limits for a family of non-Hermitian random matrix ensembles}
\author{Martin Bender}\thanks{Supported by grant KAW 2005.0098 from the Knut and Alice Wallenberg Foundation}   
\address{Department of mathematics, KTH, S-100 44, Stockholm, Sweden}
\email{mbender@math.kth.se}
\maketitle

\begin{abstract}

A family of random matrix ensembles
interpolating between the GUE and the Ginibre ensemble of $n\times n$
matrices with iid centered complex Gaussian 
entries is considered. The asymptotic spectral distribution 
in these models is uniform in an ellipse in the complex plane, 
which collapses 
to an interval of the real line as the degree of 
non-Hermiticity diminishes. 
Scaling limit theorems are proven for the eigenvalue point process at the 
rightmost edge of the spectrum, and it is shown that a non-trivial 
transition occurs between Poisson and Airy point process statistics 
when the ratio of the axes of the supporting 
ellipse is of order $n^{-1/3}$. In this regime, the family of 
limiting probability
distributions of the maximum of the real parts of the 
eigenvalues interpolates 
between the Gumbel and Tracy-Widom distributions.  

\end{abstract}

\section{Introduction of the model}

\subsection{Introduction}
The local statistics at the edge of the spectrum are of fundamental 
interest in the study of random matrix models. For the Gaussian unitary 
ensemble (GUE), defined as the probability measure
\begin{equation}
\ud \mathbb{P}^{1}_n(H)
=\left(\frac{n^{n^2}}{2^{n}\pi^{n^2}}\right)^{1/2}\exp\left\{-\frac{n}{2} 
\operatorname{Tr} H^2\right\}\ud H 
\end{equation} 
on the space $\mathcal{H}_n\cong \mathbb{R}^{n^2}$ of Hermitian matrices, where $\ud H$ denotes Lebesgue measure, 
the celebrated work of Tracy and Widom, \cite{TracyWidomLevel},  
shows that the distribution of  the 
largest eigenvalue,
properly rescaled, converges  to $F_{TW}$, 
the Tracy-Widom distribution, which is 
expressible in terms of a solution to the Painlev\'e II differential 
equation. Their results
have been generalized to other models and it is conjectured that $F_{TW}$ is the universal limiting extreme value law for large classes 
of random matrix ensembles. 
This universality conjecture is supported by the fact that the 
Tracy-Widom distribution has appeared in a number of seemingly 
unrelated contexts, see for instance \cite{Baik}, \cite{JohanssonDisc} and
\cite{JohanssonPoly}; in some sense it is a natural extreme 
value distribution. 

In contrast, the classical extreme value theorem for a set 
of \emph{independent} identically distributed (iid)
random variables, say $X_i\in N(0,1/2)$, states that the distribution
of their appropriately shifted and rescaled maximum converges 
to the \emph{Gumbel} distribution, 
$F_{G}(x):=e^{-e^{-x}}$,
\begin{equation*}
\lim_{n\to \infty}\mathbb{P}\left[\frac{\max_{1\leq i\leq n}\{X_i\}-c_n}{a_n}
\leq t\right]= F_{G}(t),
\end{equation*}
where
\begin{equation*}
c_n=\sqrt{\log n}-\frac{\log(4\pi \log n)}{4\sqrt{\log n}},
\end{equation*}
and
\begin{equation*}
a_n=\frac{1}{2\sqrt{\log n}}.
\end{equation*}

Although the edge behaviours of these models have received much less
 attention in the literature, there are random matrix ensembles 
without symmetry conditions imposed so that the limiting spectral 
density is no longer supported on a line, but some other set in the 
complex plane. The simplest example of such a model is the classical
\emph{Ginibre ensemble} of matrices with independent centered 
complex Gaussian entries of variance $n^{-1}$, introduced in 
\cite{Ginibre}, for which the asymptotic spectral distribution is 
uniform in (and vanishes outside) the unit disc. There are 
several possible notions of ``extremeness'' that might be 
considered in this model; Rider \cite{Rider} shows that the 
spectral radius, that is, the maximum absolute value of the 
eigenvalues, is asymptotically Gumbel-distributed after a shifting 
and rescaling much the same as for the case of extremes of 
independent random variables. Intuitively, this can be understood 
as reflecting the fact that the few eigenvalues with largest 
modulus do not in general lie close to or feel the 
repulsion of each other,  and thus behave as if they 
were independent. 
Rider's method relies heavily on the radial symmetry of the 
problem; a closely related problem which generalizes easier 
to the family of models to be considered here is that of determining the 
distribution of the maximum 
of the \emph{real parts} of the eigenvalues instead of their 
absolute values.

\subsection{The ellipse ensemble}
 The fundamental difference between the edge statistics of the GUE 
and those of the Ginibre ensemble corresponds to the distinction between 
random matrix and Poisson statistics which has attracted a
great deal of attention from the mathematical 
physics community in 
connection with quantum chaos. In particular, it is illuminating to 
study families of models that exhibit a transition between 
the two types of behaviour, see e.g. \cite{JohanssonGumbel}. 
 Based on the observation
that a Ginibre matrix $A$ can equivalently be considered as a sum 
$A=(H_1+iH_2)/\sqrt{2}$ where $H_1$ and $H_2$ are two independent 
GUEs, a natural family of random matrix ensembles 
interpolating between the GUE and Ginibre ensembles can be defined as the 
distribution of 
 \begin{equation}
A=\sqrt{\frac{(1+\tau)}{2}}H_1+i\sqrt{\frac{(1-\tau)}{2}}H_2,
\end{equation} 
for $\tau\in [0,1]$. The parameter 
$\tau$ can thus be thought of as determining the extent to which $A$ fails
 to be Hermitian. 
In \cite{Fyodorov}, the authors define this \emph{ellipse ensemble} 
explicitly as the probability measure $\mathbb{P}^{\tau}_n$ on the 
space $\mathcal{M}_n$ of complex $n\times n$ matrices given, for $\tau< 1$, by
\begin{equation}
\ud \mathbb{P}^{\tau}_n(A)
=\left(\frac{n}{\pi\sqrt{(1-\tau^{2})}}\right)^{n^2}
\exp\left\{-\frac{n}{(1-\tau^{2})} 
\operatorname{Tr} (AA^{\dag}-\tau\operatorname{Re}{A^2})\right\}\ud A,
\end{equation} 
where 
$\ud A$ is Lebesgue measure on  $\mathcal{M}_n\cong \mathbb{R}^{2n^2}$. 
 It is easy to see that the cases $\tau=0$ and $\tau \to 1$ 
correspond to the Ginibre and Gaussian unitary ensembles 
respectively.
The successful analysis 
of this model depends crucially on the fact that, for every 
$\tau\in(0,1)$, the induced eigenvalue measure
\begin{multline}
\ud {{\mathbb{P}}_n^{\tau}}'(\zeta_1,\ldots, \zeta_n)\\
= \frac{1}{\mathcal{Z}_N^{\tau}}\prod_{1\leq j<k\leq n}|\zeta_k-\zeta_j|^2\exp\left\{-\frac{n}{(1-\tau^2)}\sum_{j=1}^n\left(|\zeta_j|^2-\tau\operatorname{Re}\zeta_j^2\right)\right\}\ud ^n \zeta_j,
\end{multline} 
where $\mathcal{Z}_N^{\tau}$ is a 
normalizing constant, can be considered 
as the distribution of a point process $Z_n^{\tau}$ on $\mathbb{R}^2$ 
($\cong\mathbb{C}$) with determinantal correlation functions 
given by the correlation kernel
\begin{multline}\label{defkernel} 
K_n^{\tau}\left((\xi_1,\eta_1),(\xi_2,\eta_2)\right)\\
=\frac{n}{\sqrt{\pi(1-\tau^2)}}
\sum_{k=0}^{n-1}\tau^k h_k\left(\sqrt{\frac{n}{2\tau}}\zeta_1\right)
h_k\left(\sqrt{\frac{n}{2\tau}}\overline{\zeta}_2\right)
\exp\left\{-\frac{n}{2}\left(\frac{\xi_1^2+\xi_2^2}{(1+\tau)}+\frac{\eta_1^2+\eta_2^2}{(1-\tau)}\right)\right\}.
\end{multline}
Here $\zeta_j=\xi_j+i\eta_j$, and $h_k$ are the normalized 
Hermite polynomials, 
that is, the orthonormal polynomials with respect to the measure
$e^{-x^2}\ud x$ on $\mathbb{R}$. (See Section 
\ref{determinantalpointprocesses} for definitions related 
to determinantal point processes.)
Furthermore, the rescaled Hermite polynomials appearing in 
(\ref{defkernel}) are orthogonal in the whole complex plane
 with respect to the corresponding exponential weight.
The orthogonal polynomial 
techniques employed in the analysis of the GUE therefore 
become available for the ellipse ensemble as well, although 
considerable technical difficulties remain, since the 
Christoffel-Darboux formula no longer applies to simplify the 
sum of products of \emph{rescaled} Hermite polynomials.
 It is proven in \cite{Fyodorov} e.g. that for fixed $\tau\neq 1$, 
the limiting eigenvalue density is constant in the ellipse 
$\{(x,y):x^2/(1+\tau)^2+y^2/(1-\tau)^2\leq 1\}$, where $z=x+iy$. 
The authors  
also discover a regime of ``weak non-Hermiticity'', when
$\tau=\tau_n=1-\alpha^2/(2n)$ for a new parameter $0<\alpha<\infty$, 
in which a transition occurs in the local bulk statistics 
between the sine kernel behaviour of the GUE and the scaled
 correlation kernel of the Ginibre ensemble. In \cite{Verbaarschot}, 
it is shown heuristically that 
in this regime the local statistics near the rightmost edge
of the spectrum are given by a correlation kernel which factorizes 
as the Airy kernel in the $\xi$-variables multiplied by a 
Gaussian weight in the $\eta$-direction; thus the model 
essentially behaves like the GUE at the edge in this regime.

However, as the present work will show, there is another 
transitional regime, for $(1-\tau_n) \sim n^{-1/3}$, when 
the \emph{edge} statistics interpolate between those of 
the GUE and the Ginibre ensemble. Let 
$\sigma_n=n^{1/6}\sqrt{(1-\tau_n)}$. It is proven in Theorem 
\ref{main} that if $\sigma_n\to 0$, the rescaled eigenvalues 
near the rightmost edge of the spectrum converge 
to a point process on $\mathbb{R}^2$
which can be interpreted as the Airy point process in the 
$x$-direction with each particle subject to an iid centered 
Gaussian displacement in the $y$-direction.  Furthermore, 
the distribution of the maximum of the  real parts of 
the eigenvalues  converges to the Tracy-Widom distribution. 
The process is thus essentially one-dimensional and 
coincides with the corresponding point process at the edge 
of the GUE. (This is what is found in \cite{Verbaarschot} 
for the particular choice 
$\sigma_n \propto n^{-1/3}$.)
If  $\sigma_n\to  \infty$ the edge point process converges to a 
Poisson process on $\mathbb{R}^2$ with intensity  $\pi^{-1/2}e^{-x-y^2}$,
and in addition
the maximum of the real parts of the eigenvalues
converges in distribution to a Gumbel random variable. 
In the intermediate regime when 
$\sigma_n\to \sigma\in \mathbb{R}^+$ a new non-trivial 
interpolating point process on $\mathbb{R}^2$ arises in 
the $n\to \infty$ limit. It is again a determinantal 
point process with a rightmost particle almost surely,
but the correlation kernel no longer 
factorizes.

The edge behaviour in the various regimes can be interpreted
as follows: When $\sigma_n \to \infty$, the supporting ellipse 
collapses to an interval of the real line so slowly that the 
imaginary parts of the eigenvalues near the edge are of order
much greater than the spacing of their real parts. Therefore 
the eigenvalues are not close and do not interact, so in the 
appropriate scaling limit they behave as a Poisson process. 
On the other hand, when $\sigma_n \to 0$, the imaginary parts 
of the eigenvalues near the edge are negligible compared to 
the spacing of their real parts; hence they behave essentially 
like the particles in the Airy point process with independent 
fluctuations in the $y$-direction. In contrast, the case 
$\sigma_n\to \sigma$ corresponds precisely to the imaginary 
parts of the eigenvalues being of the same order of magnitude 
as the spacings of their real parts, and there is a non-trivial 
interaction between eigenvalues even though their 
displacement away from the real axis is no longer negligible.   

The proof relies in an essential way on 
finding, in the absence of a Christoffel-Darboux formula, 
a manageable representation of the sum of products of 
Hermite polynomials in (\ref{defkernel}). This is provided 
by the double integral formula of Lemma \ref{hermite}. 
The asymptotics of the correlation kernel can then be 
analyzed by means of saddle point techniques.    

Although  a very specific family of models, particularly accessible 
to analysis, is studied here,  it is reasonable 
to conjecture, in analogy with the universality conjectures 
for random matrix ensembles with purely real spectrum, that 
the type of transition and scale at which it occurs has a 
universal character that should be found in much greater 
generality for random matrix models in which the support of 
the spectral density collapses to a one-dimensional subset 
of the complex plane. Examples of other non-Hermitian ensembles
are provided by the model introduced in \cite{Wiegmann} and 
analyzed rigourously in \cite{Hedenmalm}.

\newpage

\section{Formulation of results}

\subsection{Preliminaries on determinantal point processes}
\label{determinantalpointprocesses}
For a comprehensive account of the theory of point processes, 
see \cite{Daley}. The material relating to determinantal processes 
can be found e.g. in \cite{Johansson}.

Let $\Lambda$ be a complete separable metric space. A Borel 
measure $\mu$ on
$\Lambda$ is a \emph{boundedly finite simple counting measure} 
if it takes 
non-negative integer values on bounded sets and 
$\mu(\{x\})\in\{0,1\}$ for every $x\in \Lambda$.
The set $\mathcal{N}(\Lambda)$  of all 
boundedly finite simple counting 
measures on $\Lambda$ can be identified with the family of sets 
$x=\{x_i\}\subset \Lambda$ such that $x_i\neq x_j$ if 
$i\neq j$ and $|x\cap B|<\infty$ for every bounded $B\subset \Lambda$.
Define the smallest $\sigma$-algebra  on $\mathcal{N}(\Lambda)$ 
such that the mappings 
$\mathcal{N}(\Lambda) \ni \mu  \mapsto \mu(A)$ are measurable
for each Borel set $A\in\Lambda$, making $\mathcal{N}(\Lambda)$ 
a measurable space. A \emph{point process} $X=\{x_j\}$ on 
$\Lambda$ is a random element of $\mathcal{N}(\Lambda)$. 
 
Let $\lambda$ be a reference measure on $\Lambda$ (in our case, 
$\Lambda$ will be $\mathbb{R}$ or $\mathbb{R}^2$  and $\lambda$ 
Lebesgue measure). 
If $X$ is a point process on $\Lambda$ and, for some $n\geq 1$,
there exists a measurable function $\rho_n:\Lambda^n\to \mathbb{R}$ 
such that for every bounded measurable function $\varphi:\Lambda^n \to \mathbb{R}$,
\begin{equation}
\label{defcorrelation1}
\mathbb{E}\left[\sum_{x_{k_j}\in X}
\varphi(x_{k_1},\ldots,x_{k_n})\right]
=\int_{\Lambda^n}\varphi(\xi_1,\ldots,\xi_n)
\rho_n(\xi_1,\ldots,\xi_n)\ud^n\lambda(\xi),
\end{equation}
then $\rho_n$ is called an \emph{$n$-point correlation function}
 of $X$. (The sum on the left hand side of (\ref{defcorrelation1}) 
is over all $n$-tuples of distinct points in $X$, 
including permutations.)  It can be shown
that if $X$ has correlation functions 
$\rho_n$ for every $n\geq1$ (and the product over an empty 
index set is $1$ by definition), then
\begin{equation}
\label{defcorrelation2}
\mathbb{E}\left[\prod_j(1+\phi(x_j))\right]
=\sum_{n=0}^{\infty}\frac{1}{n!}\int_{\Lambda^n}\prod_{j=1}^n\phi(\xi_j)
\rho_n(\xi_1,\ldots,\xi_n)\ud^n\lambda (\xi)
\end{equation}
for every bounded measurable function $\phi$ with bounded support. 
(Here the product on the left hand side is over all particles in 
the point process; since there are only finitely many particles 
in each bounded set the product is finite for each realization of $X$).

A point process for which all correlation functions exist
 and are of the form 
\begin{equation*}
\rho_n(\xi_1,\ldots,\xi_n)
=\det\left(K(\xi_i,\xi_j)\right)_{i,j=1}^n 
\end{equation*} 
for some measurable function 
$K:\Lambda^2\to \mathbb{C}$ (a \emph{correlation kernel}) 
is called a \emph{determinantal (point) process}.
 The correlation kernel $K$ is said to be \emph{Hermitian} if
 \begin{equation*}
K(\xi_2,\xi_1)=\overline{K(\xi_1,\xi_2)}
\end{equation*}
 for all $(\xi_1,\xi_2)\in \Lambda^2$. 

Weak convergence in $\mathcal{N}(\Lambda)$ of a sequence 
$\{X_n\}$ of point processes to a limit $X$ is equivalent 
to weak convergence of the finite dimensional distributions, 
meaning that for any finite family
 $\{A_i\}_{i=1}^k$ of bounded disjoint Borel sets in 
$\Lambda$ and integers $m_i$,     
\begin{equation*}
\lim_{n\to \infty}\mathbb{P}_n\left(\bigcap_{i=1}^k\left[|X_n\cap A_i|=m_i\right]\right)
=\mathbb{P}\left(\bigcap_{i=1}^k\left[|X\cap A_i|=m_i\right]\right). 
 \end{equation*}  
 
Consider a determinantal point process 
$Z=\{z_j\}=\{(x_j,y_j)\}$ on $\mathbb{R}^2$. 
If there is a $t \in\mathbb{R}$ such that 
$\mathbb{E}\left[|Z\cap \left((t,\infty)\times \mathbb{R}\right)|\right]<\infty$,
 $Z$ is said to have a \emph{rightmost} or \emph{last} 
particle almost surely. The \emph{last particle distribution function},
 $F$, of $Z$ is then defined as 
\begin{equation}
F(t)=\mathbb{P}\left[\left|Z\cap\left((t,\infty)\times \mathbb{R}\right)\right|=0\right].
\end{equation}

To prove the limit theorems of this paper we will appeal to the 
following lemma, which is a fairly standard result. For completeness, a proof is given in an appendix.

\begin{lemma}\label{convergenceoflastparticle}
For every positive integer $n$, let  $Z_n$ be a determinantal point 
process on $\mathbb{R}^2$ with Hermitian correlation kernel $K_n$. 
 Suppose that $K_n(\zeta_1,\zeta_2)\to K(\zeta_1,\zeta_2)$ 
as $n\to \infty$ for 
every 
$(\zeta_1,\zeta_2)\in \mathbb{R}^4$, and that 
there is a measurable function $B:\mathbb{R}^2 \to \mathbb{R}$ with 
$B(\zeta)\geq|K_n(\zeta,\zeta)|$  for every $n$, such that for each 
$\xi_0\in \mathbb{R}$
\begin{equation*} 
\int_{(\xi_0,\infty)\times \mathbb{R}}B(\zeta)\ud \zeta=C_{\xi_0}<\infty.
\end{equation*} 
Then $Z_n$ converges weakly as $n\to \infty$ to a determinantal 
point process $Z$ 
with correlation kernel $K$, $Z_n$ has a last particle almost 
surely with distribution 
\begin{equation}\label{lastparticle}
F_n(t)
=\sum_{r=0}^{\infty}\frac{(-1)^r}{r!}\int_{\left((t,\infty)\times \mathbb{R}\right)^r}\det (K_n(\zeta_j,\zeta_k))_{j,k=1}^{r}\ud^r\zeta,
\end{equation}
 and $F_n$  converges to the last particle distribution $F$ of $Z$ as $n\to \infty$.
\end{lemma}

\subsection{The interpolating process}
The \emph{Airy kernel} is the function 
$K_{A}:\mathbb{R}^2\to \mathbb{R}$ defined by 
\begin{equation*}
K_{A}(\xi_1,\xi_2)
=\frac{\operatorname{Ai}(\xi_1)\operatorname{Ai}'(\xi_2)
-\operatorname{Ai}'(\xi_1)\operatorname{Ai}(\xi_2)}{\xi_1-\xi_2}
\end{equation*}
 if $\xi_1\neq \xi_2$ and by continuity if $\xi_1=\xi_2$. 
Here $\operatorname{Ai}:\mathbb{R}\to\mathbb{R}$ is the 
$C^{\infty}$ \emph{Airy function}, 
\begin{equation*}
\operatorname{Ai}(t)
=\frac{1}{2\pi}\int_{\gamma}e^{\frac{i}{3} u^3+i u t}\ud u \in \mathbb{R}, 
\end{equation*}
and $\gamma$ is the contour $t\mapsto t+i \delta$ for some $\delta>0$,
 independently of the choice 
of $\delta$ by Cauchy's theorem. 
We will use the double integral representation 
\begin{equation}\label{doubleintegralairy}
K_{A}(\xi_1,\xi_2)
=\frac{1}{4\pi ^2}\iint_{\gamma,\gamma}
\frac{e^{\frac{i}{3}v^3+i\xi_2v+\frac{i}{3}u^3+i\xi_1u}}{i(u+v)}\ud u \ud v
\end{equation} 
of the Airy kernel.
It is well known  (see \cite{TracyWidomLevel} and \cite{Forrester}) 
that $K_{A}$ 
defines a trace class integral operator on $L^2(t,\infty)$ for 
any real $t$ and that it is the correlation kernel of a determinantal 
point process $X_A$ on $\mathbb{R}$, the \emph{Airy point process}. 
$X_A$ has 
a last particle almost surely, the distribution of which 
is known as the \emph{Tracy-Widom distribution}, given by the Fredholm
 determinant 
\begin{equation*}
F_{TW}(t):=\det(I-K_{A})_{L^2(t,\infty)}
=\sum_{k=0}^{\infty}\frac{(-1)^k}{k!}\int_{(t,\infty)^k}
\det (K_{A}(\xi_i,\xi_j))_{i,j=1}^{k}\ud^k \xi.
\end{equation*} 
Alternatively, 
\begin{equation}\label{painleve}
F_{TW}(t)=\exp\left\{-\int_t^{\infty}(x-t)q^2(x)\ud x\right\},
\end{equation}
where $q$ is the unique solution to the Painlev\'e II equation 
\begin{equation*}
q''(t)=tq(t)+2q(t)^3
\end{equation*} with the asymptotic behaviour
\begin{equation*}
q(t)\sim \operatorname{Ai}(t) \textrm{ as } t\to \infty. 
\end{equation*}
 
We now construct a point process $Z_A$ on $\mathbb{R}^2$ 
which, intuitively, to each realization $\{x_i\}$ of $X_A$ associates 
a realization $\{(x_i,y_i)\}$ of $Z_A$, where 
the $y_i$ are picked independently from $N(0,1/2)$. More precisely, 
for every $k\geq 1$, $Z_A$ should have correlation functions 
\begin{align}\label{productkernel}
\rho_k\left(\zeta_1,\ldots,\zeta_k\right)
&:=
\det (K_A(\xi_i,\xi_j))_{i,j=1}^{k}\prod_{i=1}^k
\frac{e^{-\eta_i^2}}{\sqrt{\pi}}\nonumber\\
&=\det \left(\frac{e^{-\frac{1}{2}(\eta_i^2+\eta_j^2)}}
{\sqrt{\pi}}K_A(\xi_i,\xi_j)\right)_{i,j=1}^{k},
\end{align}
making $Z_A$ a determinantal process as well.

A Poisson process $Z_P$ on $\mathbb{R}^2$ with intensity 
$\pi^{-1/2}e^{-\xi-\eta^2}$ 
can be viewed as a degenerate determinantal point process with 
correlation kernel
\begin{equation}
M_{P}(\zeta_1,\zeta_2)
=\delta_{\zeta_1\zeta_2}\frac{e^{-\xi_1-\eta_1^2}}{\sqrt{\pi}}
=\left\{ \begin{array}{ll}\pi^{-1/2}e^{-\xi_1-\eta_1^2}
\textrm{ if }\zeta_1=\zeta_2\\ 0 \textrm{ otherwise.}\end{array}\right.
\end{equation} 
To stress the analogy with the construction above, define
\begin{equation}
M_{P1}(\zeta_1,\zeta_2)
=\frac{e^{-\frac{1}{2}(\eta_1^2+\eta_2^2)}}{\sqrt{\pi}}K_{P}(\xi_1,\xi_2),
\end{equation}
where $K_{P}(\xi_1,\xi_2)=\delta_{\xi_1\xi_2}e^{-\xi_1}$ is the 
correlation 
kernel of a Poisson process $X_P$ on $\mathbb{R}$ with 
intensity $e^{-\xi}$.
$X_P$ too has a last particle almost surely, which is 
Gumbel-distributed, since 
\begin{multline*}
\sum_{n=0}^{\infty}\frac{(-1)^n}{n!}\int_{(t,\infty)^n}
\det (K_{P}(\xi_i,\xi_j))_{i,j=1}^{n}\ud^n\xi\\
=\sum_{n=0}^{\infty}\frac{(-1)^n}{n!}
\int_{(t,\infty)^n}\prod_{i=1}^ne^{-\xi_i}\ud^n \xi
=\sum_{n=0}^{\infty}\frac{(-1)^n}{n!}\left(e^{-t}\right)^n
=F_{G}(t).
\end{multline*}
For any $n\geq 1$, 
\begin{equation*}
\det (M_{P}(\zeta_j,\zeta_k))_{j,k=1}^{n}
=\det (M_{P1}(\zeta_j,\zeta_k))_{j,k=1}^{n} \textrm{ a.e. in } 
\mathbb{R}^{2n},
\end{equation*} 
so $M_{P1}$ can equivalently be chosen as a correlation kernel of $Z_P$. 
Similarly,
\begin{equation}
M_{P2}(\zeta_1,\zeta_2)
=\delta_{\eta_1\eta_2}
\frac{e^{-\eta_1^2-\frac{1}{2}(\xi_1^2+\xi_2^2)}}{\sqrt{\pi}}
\end{equation}
is also an equivalent correlation kernel of $Z_P$. The following proposition
sums up these observations.

\begin{proposition}\label{limitkernels}
The determinantal point 
processes $Z_{P}$ and $Z_{A}$ on $\mathbb{R}^2$ with correlation kernels
$M_{P}$ (or equivalently $M_{P1}$ or $M_{P2}$) and  
\begin{equation}
M_{A}(\zeta_1,\zeta_2)
=\frac{e^{-\frac{1}{2}(\eta_1^2+\eta_2^2)}}{\sqrt{\pi}}K_{A}(\xi_1,\xi_2)
\end{equation} 
respectively, both have last particles almost surely, 
with distribution functions $F_{G}$ and $F_{TW}$ respectively.
\end{proposition}

\begin{proof} If $K$ is a Hermitian correlation kernel of a determinantal
 point process on $\mathbb{R}$ satisfying 
\begin{equation*}
\int_t^{\infty}K(\xi,\xi)\ud \xi<\infty,
\end{equation*}
it is immediate that 
$M(\zeta_1,\zeta_2)=\pi^{-1/2}e^{-\frac{1}{2}(\eta_1^2+\eta_2^2)}K(\xi_1,\xi_2)$ 
satisfies the corresponding conditions of Lemma \ref{lastparticlelemma}. 
Since the correlation functions factorize by (\ref{productkernel}),
the $\eta_i$-variables may be integrated out in the expression 
(\ref{lastparticle2}) for the last particle distribution.
\end{proof}

The following theorem asserts the existence of
 a family of determinantal point processes 
$\{Z_{\sigma}\}_{\sigma>0}$ on $\mathbb{R}^2$ which, appropriately 
rescaled, interpolates between  $Z_{A}$ and $Z_{P}$.

\begin{theorem}\label{interpolatingkernel}\emph{(Interpolating process)} 
For each $\sigma\in [0,\infty)$ there exists a determinantal point 
process 
\begin{equation*}
Z_{\sigma}=\left\{\left(x_j,y_j\right)\right\} 
\end{equation*}
 on $\mathbb{R}^2$ with correlation kernel 
\begin{equation}
M_{\sigma}(\zeta_1,\zeta_2)
=\frac{1}{4\pi^{5/2}}\int_{\gamma_1}\int_{\gamma_2}
\frac{e^{-\frac{1}{2}(\sigma v-\eta_2)^2+\frac{i}{3}v^3+i\xi_2v 
-\frac{1}{2}(\sigma u+\eta_1)^2+\frac{i}{3}u^3+i\xi_1u}}{i(u+v)}\ud u \ud v,
\end{equation}
where $\zeta_j=(\xi_j,\eta_j)$ and $\gamma_j$ is the contour 
$t \mapsto \gamma_j(t)=t+i\delta_j$, independently of the choice of $\delta_j>0$. 

Define the rescaled point process  
\begin{equation*}
\tilde{Z}_{\sigma}
=\left\{\left(\frac{x_j-c_{\sigma}}{a_{\sigma}},
\frac{y_j}{b_{\sigma}}\right)\right\},
\end{equation*} 
where
\begin{equation*} a_{\sigma}=\frac{\sigma}{\sqrt{6\log\sigma}}, 
\end{equation*}
 \begin{equation*} b_{\sigma}
=\frac{\sigma^{3/2}}{(6\log\sigma)^{1/4}},
\end{equation*}
and
\begin{equation*}
c_{\sigma}=a_{\sigma}\left(3\log\sigma-\frac{5}{4} \log (6\log \sigma)-\log(2 \pi)\right).
\end{equation*}
The processes $Z_{\sigma}$, appropriately rescaled, interpolate between 
$Z_{A}$ and $Z_{P}$ in the sense that  $Z_0=Z_A$  and  
$\tilde{Z}_{\sigma}$ converges weakly in $\mathcal{N}(\mathbb{R}^2)$ 
to $ Z_{P}$ as $\sigma \to \infty$.
For any fixed $\sigma$, $Z_{\sigma}$ has a last particle almost 
surely, with distribution function 
\begin{align}
\label{lastparticlesigma}F_{\sigma}(t)
&=\mathbb{P}_{\sigma}\left[\left|Z_{\sigma}
\cap\left((t,\infty)\times\mathbb{R}\right)\right|=0 \right]\nonumber\\
&=\det(I-M_{\sigma})_{L^2\left((t,\infty)\times \mathbb{R}\right)}\nonumber\\
&=\sum_{r=0}^{\infty}\frac{(-1)^r}{r!}
\int_{\left((t,\infty)\times \mathbb{R}\right)^r}
\det (M_{\sigma}(\zeta_j,\zeta_k))_{j,k=1}^{r}\ud^r\xi \ud^r\eta,
\end{align}
and furthermore 
\begin{equation}
F_{\sigma}(c_{\sigma}+a_{\sigma}t) \to F_{G}(t)
\textrm{ as } \sigma \to \infty.
\end{equation}
\end{theorem}
\begin{remark} To the author's knowledge, the processes $Z_{\sigma}$  
 have not been studied previously. In particular, 
we are not aware of any generalization to  
$\sigma>0$ of the representation (\ref{painleve}), for the case $\sigma=0$,
of the Fredholm determinant (\ref{lastparticlesigma}).  
\end{remark}
The proof is deferred to Section \ref{interpolationproof}.

\subsection{Edge scaling limits}
This section contains the statement of 
the main result, giving scaling limits around the 
rightmost particle of the eigenvalue point process $Z_n^{\tau_n}= \{z_j\}_{j=1}^n=\{(x_j,y_j)\}_{n=1}^n$
of the ellipse ensemble, defined formally as the 
measurable function 
\begin{equation*} Z_n^{\tau_n}:
(\mathcal{M}_n,\mathbb{P}_n^{\tau_n})\to \mathcal{N}(\mathbb{R}^2)
\end{equation*}
mapping $A\in \mathcal{M}_n$ to the set $\{z_j\}_{j=1}^n$ 
of eigenvalues of $A$ (with probability one, there are $n$ 
distinct (complex)
eigenvalues and we assume this to be the case).
 Recall that $Z_n^{\tau_n}$ is 
determinantal with correlation kernel $K_n^{\tau_n}$.

\begin{theorem}\label{main} 
Let  $\{\tau_n\}_{n=1}^{\infty}\subseteq [0,1)$
be given 
and put  
$\sigma_n=n^{1/6}\sqrt{(1-\tau_n)}$. For the choices 
of scaling parameters $\tilde{a}_n$, $\tilde{b}_n$ and 
$\tilde{c}_n$ specified below,  define the rescaled edge eigenvalue point process
\begin{equation*}
\tilde{Z}_n^{\tau_n}
=\left\{\left(\frac{x_j-\tilde{c}_n}{\tilde{a}_n},
\frac{y_j}{\tilde{b}_n}\right)\right\}_{j=1}^n,
\end{equation*} 
and let
\begin{align*}
F_{n}^{\tau_n}(t)
&=\mathbb{P}_n^{\tau_n}\left[\left|\tilde{Z}_n^{\tau_n}\cap\left((t,\infty)\times\mathbb{R}\right)\right|=0 \right]\nonumber\\
&= \mathbb{P}_{n}^{\tau_n}\left[\frac{\max_{1\leq j\leq n}\{x_j\}
-\tilde{c}_n}{\tilde{a}_n}\leq t\right]
\end{align*}
be the last particle distribution of $\tilde{Z}_n^{\tau_n}$. 
\begin{itemize}
\item[(i)]\label{main1} Suppose $\sigma_n \to \infty$ as $n\to \infty$.
Choose
\begin{equation}\tilde{a}_n\label{gumbelreal}
={\hat{\tau}}^{1/2}_n\frac{\sigma_n n^{-2/3}}{\sqrt{6\log \sigma_n}},
\end{equation}

\begin{equation}
\label{gumbelimaginary}\tilde{b}_n
=\hat{\tau}_n^{-1/4}\frac{\sigma_n^{5/2}n^{-2/3}}{(6\log \sigma_n)^{1/4}},
\end{equation}
and
\begin{equation}
\label{gumbelmean}\tilde{c}_n
=2\hat{\tau}_n+ \tilde{a}_n\left(3\log \sigma_n-\frac{5}{4}\log(6\log \sigma_n)
-\log(2\pi\hat{\tau}_n^{3/4})\right),
\end{equation}
where $\hat{\tau}_n:=(1+\tau_n)/2$.
Then 
$\tilde{Z}_n^{\tau_n}$ 
converges weakly to 
$Z_{P}$ 
and 
$F_{n}^{\tau_n}(t)$ converges to $F_{G}(t)$  as $n\to \infty$.
\item[(ii)]\label{main2} Suppose 
$\sigma_n\to \sigma \in [0,\infty)$ as $n\to \infty$.
Choose  $\tilde{a}_n=n^{-2/3}$, $\tilde{b}_n=\sigma_n n^{-2/3}$ and 
$\tilde{c}_n=(1+\tau_n)$.
Then $\tilde{Z}_n^{\tau_n}$ converges weakly to $Z_{\sigma}$ 
and $F_{n}^{\tau_n}(t)$ converges to $F_{\sigma}(t)$ as $n\to \infty$.

In particular, if $\sigma=0$, $\tilde{Z}_n^{\tau_n}$ converges weakly to $Z_{A}$ 
and $F_{n}^{\tau_n}(t)$ converges to  $F_{TW}(t)$. 
\end{itemize}
\end{theorem}
\begin{remark} Although the case $\tau_n=0$ has not been explicitly 
included in the calculations, a similar (simpler) argument can be made
 in this case (see the comment following (\ref{integralkernel})); as can be expected from the fact that the ellipse
 ensemble for constant $\tau\in (0,1)$ is just a trivial rescaling
 of the Ginibre ensemble (see \cite{Fyodorov}), formally inserting 
$\tau_n=0$ in (\ref{gumbelreal}) through 
(\ref{gumbelmean}) indeed gives the correct limit theorem
for the pure Ginibre ensemble as well.
\end{remark} 
\begin{remark} In view of the scaling limits in Theorems 
\ref{interpolatingkernel} and \ref{main} (ii), 
the scaling limit in Theorem \ref{main} (i) 
is the only reasonable candidate 
(as long as $\tau_n\to 1$); together these theorems in essence 
assert that $\tilde{Z}_n^{\tau(\sigma)}$ converges to  
$Z_P$ regardless of 
whether $\sigma$ tends to infinity with $n$ or the 
$n$ limit is taken first.
\end{remark}
\section{Proof of  Theorem \ref{interpolatingkernel}}\label{interpolationproof}
The asymptotic notation defined in the first paragraph of Section \ref{mainproof} is used in this section as well, with $n$ replaced by $\sigma$.
\begin{proof}[Proof of Theorem \ref{interpolatingkernel}]
The existence of $Z_{\sigma}$, and the fact that it has a last 
particle with distribution (\ref{lastparticlesigma}), follows by Lemma 
\ref{convergenceoflastparticle} from the convergence of  $Z_n^{\tau_n}$  
when $(\tau_n)_{n=1}^{\infty}$ is a sequence such that  
$n^{1/6}\sqrt{(1-\tau_n)} \to \sigma$ as $n\to \infty$.
This is established in the proof of Theorem \ref{main}.
Note that $Z_0$ coincides with $Z_A$, by (\ref{doubleintegralairy}).

To prove weak convergence of $\tilde{Z}_\sigma$ to $Z_P$ using
Lemma \ref{convergenceoflastparticle}, we first prove point-wise convergence of 
the correlation kernels.  
From the definition of correlation functions 
(\ref{defcorrelation1}), 
it is clear that the rescaled point process $\tilde{Z}_{\sigma}$ is 
determinantal with correlation kernel  
\begin{equation}
\tilde{M}_{\sigma}(\zeta_1,\zeta_2):
=a_{\sigma}b_{\sigma}M_{\sigma}((c_{\sigma}
+a_{\sigma}\xi_1,b_{\sigma}\eta_1),(c_{\sigma}
+a_{\sigma}\xi_2,b_{\sigma}\eta_2)).
\end{equation}
 We fix $(\zeta_1,\zeta_2)=\left((\xi_1,\eta_1),
(\xi_2,\eta_2)\right)\in \mathbb{R}^4$ and calculate 
the $\sigma\to \infty$ limit of $\tilde{M}_{\sigma}$ by a 
saddle point argument.
Let 
\begin{equation}
f_{\sigma}(u)=-\frac{1}{2}(\sigma u+\eta'_1)^2+\frac{i}{3}u^3+i\xi'_1u,
\end{equation}
and
\begin{equation}
g_{\sigma}(v)=-\frac{1}{2}(\sigma v-\eta'_2)^2+\frac{i}{3}v^3+i\xi'_2v,
\end{equation} 
where $\xi'_j=c_{\sigma}+a_{\sigma}\xi_j$ and $\eta'_j=b_{\sigma}\eta_j$.
Choose the saddle points  $u_0$ of $f_{\sigma}$ and $v_0$ of  
$g_{\sigma}$ satisfying
\begin{equation*}
u_0
=-\frac{\eta'_1}{\sigma}+\frac{i\xi'_1}{\sigma^2}
+\frac{2i{\eta'_1}^2}{\sigma^4}
+\frac{2\xi'_1\eta'_1}{\sigma^5}
+\mathcal{O}\left(\frac{1}{\sigma^{5/2}(\log \sigma)^{3/4}}\right),
\end{equation*} 
and
\begin{equation*}
v_0=\frac{\eta'_2}{\sigma}+\frac{i\xi'_2}{\sigma^2}
+\frac{2i{\eta'_2}^2}{\sigma^4}-\frac{2\xi'_2\eta'_2}{\sigma^5}
+\mathcal{O}\left(\frac{1}{\sigma^{5/2}(\log \sigma)^{3/4}}\right),
\end{equation*} 
respectively.
Note that
\begin{equation*}
g_{\sigma}''(v_0)
=-\sigma^2+2iv_0
=-\sigma^2\left(1+\mathcal{O}\left(\frac{1}{\sigma^{3/2}(\log\sigma)^{1/4}}\right)\right)
\end{equation*}
 and similarly
$f_{\sigma}''(u_0)=-\sigma^2+2iu_0$.
After the change of variables $s=\sigma(u-u_0)$, $t=\sigma (v-v_0)$, 
and choosing $\delta_1=\operatorname{Im} (u_0)$ and 
$\delta_2=\operatorname{Im} (v_0)$ so that the new contours of 
integration become the real axis,  the rescaled kernel can be written
\begin{equation}\label{Mkernel}
\tilde{M}_{\sigma}(\zeta_1,\zeta_2)
=\frac{a_{\sigma}b_{\sigma}}{4\pi^{5/2}\sigma^2}\iint_{\mathbb{R}^2}
\frac{e^{f_{\sigma}(s/\sigma+u_0)+g_{\sigma}(t/\sigma+v_0)}}
{i(t/\sigma+u_0+s/\sigma+v_0)}\ud s \ud t.
\end{equation}
The main contribution comes from near the saddle points; 
let $I=(-r_0,r_0)$ for some $1\ll r_0 \ll \sigma^{3/4}(\log \sigma)^{1/8}$ and
note that
\begin{multline}
\label{globalinterpolating}
\left|\frac{a_{\sigma}b_{\sigma}}{4\pi^{5/2}\sigma^2}
\iint_{\mathbb{R}\times\left(\mathbb{R}\setminus I\right)}
\frac{e^{f_{\sigma}(s/\sigma+u_0)+g_{\sigma}(t/\sigma+v_0)}}
{i(t/\sigma+u_0+s/\sigma+v_0)}\ud s \ud t\right|\\
\leq\frac{a_{\sigma}b_{\sigma}}{4\pi^{5/2}\sigma^2(\delta_1+\delta_2)}
\left|\int_{\mathbb{R}}e^{f_{\sigma}(s/\sigma+u_0)}\ud s\right|
\left|\int_{\mathbb{R}\setminus I}e^{g_{\sigma}(t/\sigma+v_0)}\ud t \right|.
\end{multline}
Now
\begin{multline*}
g_{\sigma}(t/\sigma+v_0)\\
=g_{\sigma}(v_0)+t^2\left(-\frac{1}{2}
+\frac{iv_0}{\sigma^2}\right)+ t\left(-\sigma v_0+\eta'_2
+\frac{iv_0^2}{\sigma}+\frac{i\xi'_2}{\sigma} \right)
+\frac{it^3}{3\sigma^3},
\end{multline*}
so putting
\begin{equation*}\epsilon_1=\operatorname{Re}
\left(\frac{iv_0}{\sigma^2}\right)
=\mathcal{O}\left(\frac{1}{\sigma^{3/2}(\log\sigma)^{1/4}}\right)
\end{equation*}
and
\begin{equation*}
\epsilon_2=\operatorname{Re}\left(-\sigma v_0+\eta'_2+\frac{iv_0^2}{\sigma}
+\frac{i\xi'_2}{\sigma}\right)
=\mathcal{O}\left(\frac{(\log\sigma)^{1/4}}{\sigma^{3/2}}\right)
\end{equation*}
gives
\begin{equation*}
\left|\int_{\mathbb{R}\setminus I}e^{g_{\sigma}(t/\sigma+v_0)}\ud t \right|
\leq \left|e^{g_{\sigma}(v_0)}\right|\int_{\mathbb{R}\setminus I}
e^{ t^2\left(-\frac{1}{2}-\epsilon_1\right)
+\epsilon_2t}\ud t
\leq C\frac{e^{-r_0^2/3}}{r_0}\left|e^{g_{\sigma}(v_0)}\right|.
\end{equation*}
By the same argument for $f_{\sigma}$, and 
(\ref{globalinterpolating}), it follows that
\begin{multline}\label{Mtails}
\left|\frac{a_{\sigma}b_{\sigma}}{4\pi^{5/2}\sigma^2}
\iint_{\mathbb{R}^2\setminus I^2}
\frac{e^{f_{\sigma}(s/\sigma+u_0)+g_{\sigma}(t/\sigma+v_0)}}
{i(t/\sigma+u_0+s/\sigma+v_0)}\ud s \ud t\right|\\
\leq C_1\frac{e^{-r_0^2/3}}{r_0}\frac{a_{\sigma}b_{\sigma}\left|e^{g_{\sigma}(v_0)
+f_{\sigma}(u_0)}\right|}{\sigma^2(\delta_1+\delta_2)}
\leq C\frac{e^{-r_0^2/3}}{r_0}.
\end{multline}
Since 
$\operatorname{Im}(\sigma t+\sigma^2u_0
+\sigma s+\sigma^2v_0)\asymp 2c_{\sigma}>0$, we can write
\begin{equation*}
\frac{1}{i(\sigma t+\sigma^2u_0+\sigma s+\sigma^2v_0)}
=-\int_0^{\infty}e^{ip(\sigma(t+s)+\sigma^{2}(u_0+v_0))}\ud p.
\end{equation*}
Therefore, using (\ref{Mtails}) and Fubini's theorem,
(\ref{Mkernel}) becomes
\begin{align}\label{Mkernel2}
\tilde{M}_{\sigma}(\zeta_1,\zeta_2)
=&\frac{a_{\sigma}b_{\sigma}}{4\pi^{5/2}\sigma^2}\iint_{I^2}
\frac{e^{f_{\sigma}(s/\sigma+u_0)+g_{\sigma}(t/\sigma+v_0)}}
{i(t/\sigma+u_0+s/\sigma+v_0)}\ud s \ud t+o(1)\nonumber\\
=&\frac{-a_{\sigma}b_{\sigma}e^{g_{\sigma}(v_0)+f_{\sigma}(u_0)}}
{4\pi^{5/2}}\nonumber\\
&\times \iint_{I^2}\int_0^{\infty}e^{-\frac{1}{2}(s^2+t^2)
+ip(\sigma( t+s) + \sigma^{2}(u_0+v_0))}\ud p \ud s \ud t+o(1)\nonumber\\
=&\frac{-a_{\sigma}b_{\sigma}e^{g_{\sigma}(v_0)+f_{\sigma}(u_0)}}
{2\pi^{3/2}}\int_0^{\infty}e^{-\sigma^2p^2
+i\sigma^2 p(u_0+v_0)}\ud p+o(1)\nonumber\\
=&\frac{-a_{\sigma}b_{\sigma}e^{g_{\sigma}(v_0)
+f_{\sigma}(u_0)}}{2\pi^{3/2}i\sigma^2(u_0+v_0)}+o(1),
\end{align}
where the last equality follows from the estimate
\begin{multline}
\label{PI}\left|\int_0^{\infty}
e^{-\alpha_{\sigma} r^2-(\beta_{\sigma}+i\gamma_{\sigma}) r}\ud r
-\frac{1}{\beta_{\sigma}+i\gamma_{\sigma}}\right|\\
\leq C\frac{1}{|\beta_{\sigma}+i\gamma_{\sigma}|}\left(\left(\frac{\sqrt{\alpha_{\sigma}}}{\beta_{\sigma}}\right)^2+e^{-\beta_{\sigma}/\sqrt{\alpha_{\sigma}}}\right),
\end{multline} 
obtained by integration by parts; in this case $\beta_n/\sqrt{\alpha_n}\asymp\sqrt{6\log \sigma}\gg 1$.
Factors of the form $\exp\{F(\xi'_1,\eta'_1)-F(\xi'_2,\eta'_2)\}$ can be 
freely multiplied to the correlation kernel without changing the correlation 
functions. Taking $F(\xi',\eta')=i\xi'\eta'/\sigma-i\eta'^3/3\sigma^3$ and 
calculating 
\begin{equation*}
f_{\sigma}(u_0)= 
\frac{i\xi'_1\eta'_1}{\sigma}-\frac{{\xi'_1}^2}{2\sigma^2}
-\frac{i{\eta'_1}^3}{3\sigma^3}
-\frac{\xi'_1{\eta'_1}^2}{\sigma^4}
+\mathcal{O}\left(\frac{(\log \sigma)^{3/4}}{\sigma^{3/2}}\right)
\end{equation*}
and
\begin{equation*}g_{\sigma}(v_0)
=-\frac{i\xi'_2\eta'_2}{\sigma}-\frac{{\xi'_2}^2}{2\sigma^2}
+\frac{i{\eta'_2}^3}{3\sigma^3}-\frac{\xi'_2{\eta'_2}^2}{\sigma^4}
+\mathcal{O}\left(\frac{(\log \sigma)^{3/4}}{\sigma^{3/2}}\right),
\end{equation*}
Equation (\ref{Mkernel2}) gives the equivalent kernel
\begin{align*}
{M}^*_{\sigma}(\zeta_1,\zeta_2)
&=e^{F(\xi'_1,\eta'_1)-F(\xi'_2,\eta'_2)}
\tilde{M}_{\sigma}(\zeta_1,\zeta_2)\nonumber\\
&=\frac{e^{-\frac{1}{2}(\xi_1+\xi_2)
-\frac{1}{2}(\eta_1^2+\eta_2^2)}}{\sqrt{\pi}
\left(1+\frac{i}{2}\sigma^{3/2} (6\log \sigma)^{1/4} 
(\eta_1-\eta_2)\right)}+o(1)\nonumber\\
&\to M_{P2}(\zeta_1,\zeta_2) \textrm{ as } \sigma \to \infty.
\end{align*}

It is easy to verify that 
$M_{\sigma}(\zeta_2,\zeta_1)=\overline{M_{\sigma}(\zeta_2,\zeta_1)}$ 
for all $(\zeta_1,\zeta_2) \in \mathbb{R}^4$. 
To prove weak convergence and convergence of the last particle distributions 
it suffices by Lemma \ref{convergenceoflastparticle} to show that there is
a function 
$B_{\infty}$, which is integrable on 
$(\xi_0,\infty)\times\mathbb{R}$ for every $\xi_0\in \mathbb{R}$, such 
that 
$B_{\infty}(\zeta)\geq |\tilde{M}_{\sigma}(\zeta,\zeta)|$ 
for all sufficiently large $\sigma$. 
Now
\begin{align*}
M_{\sigma}(\zeta,\zeta)
&\leq\frac{1}{4\pi^{5/2}}\int_{\gamma}\int_{\gamma}
\left|\frac{e^{-\frac{1}{2}(\sigma v-\eta)^2
+\frac{i}{3}(v^3+u^3)+i\xi (v+u) -\frac{1}{2}(\sigma u+\eta_1)^2}}{i(u+v)}\right|\ud u \ud v\nonumber\\
&=\frac{e^{\delta^2\sigma^2+\frac{2}{3}\delta^3-\eta^2-2\delta \xi}}{4\pi^{5/2}}\iint_{\mathbb{R}^2}
\frac{e^{-\frac{1}{2}(\sigma^2+2\delta)(t^2+s^2)+\sigma\eta(t-s)}}
{|s+t+2\delta i|}\ud t\ud s\nonumber\\
&\leq\frac{\exp\left\{\delta^2\sigma^2+\frac{2}{3}\delta^3-\frac{2\delta\eta^2}{\sigma^2+2\delta}-2\delta \xi\right\}}{4\pi^{3/2}\delta(\sigma^2+2\delta)}, 
\end{align*}
so choosing $\delta=1/{2a_{\sigma}}$ gives
\begin{align}
\tilde{M}_\sigma(\zeta,\zeta)
&\leq\frac{a_{\sigma}b_{\sigma}e^{\delta^2\sigma^2+\frac{2}{3}\delta^3}}
{4\pi^{3/2}\delta(\sigma^2+2\delta)} 
\exp\left\{-\frac{2\delta b_{\sigma}^2\eta^2}{\sigma^2+2\delta}-2\delta (c_{\sigma}+a_{\sigma}\xi)\right\}\nonumber\\
&=\frac{\exp\left\{\frac{\sqrt{6}(\log\sigma)^{3/2}}
{\sigma^3}\right\}}{\sqrt{\pi}\left(1+\frac{\sqrt{6\log\sigma}}
{\sigma^3}\right)}\exp\left\{-\frac{\eta^2}
{\left(1+\frac{\sqrt{6\log\sigma}}{\sigma^3}\right)}-\xi\right\}\nonumber\\
&\leq\frac{e^{1-\frac{1}{2}\eta^2-\xi}}{\sqrt{\pi}}\nonumber\\
&=:B_{\infty}(\xi,\eta)
\end{align}
for all sufficiently large $\sigma$, which concludes the proof.
\end{proof}

\section{Proof of Theorem \ref{main}}\label{mainproof}
The proof is organized as follows: In Section \ref{hermitelemma} 
an integral representation formula for the sum of products of 
Hermite polynomials appearing in (\ref{defkernel}) is derived. 
The resulting representation of the correlation kernel and a general 
discussion of the saddle point arguments used to 
calculate its asymptotics are given in 
\ref{intrepcorrkernel}. Sections \ref{proofgumbel} and 
\ref{proofairy} contain the proofs of part (i) and (ii) of Theorem \ref{main}
respectively.
 The two cases of the proof are  
preceded by Sections \ref{estimatesgumbel} and 
\ref{estimatesairy} respectively, with a series of lemmas  
estimating the integrals appearing in the correlation kernel for the two cases.

In the asymptotic estimates we use the following notation:
$C$ and  $C_i$, $i=1,2,\ldots$  are generic positive constants, 
and the occurrence of the same symbol in different chains 
of inequalities need 
not denote the same number. If $y_n>0$  and  $x_n/y_n \to 0$ as $n\to \infty$, 
we may write $x_n=o(y_n)$ or, provided $x_n\geq0$,  $x_n\ll y_n$. 
The relation $x_n=\mathcal{O}(y_n)$ means that there is a positive 
constant $C$ such that $|x_n|\leq C y_n$ for every $n$. If there is a 
positive constant $C$ such that $x_n/C<y_n<C x_n$ for every $n$, 
we write $x_n\asymp y_n$.
In the proofs of lemmas and theorems where $\zeta=(\zeta_1,\zeta_2)$ 
is explicitly stated to be fixed, 
constants may depend on $\zeta$.

\subsection{Integral formula for the Hermite polynomials}  \label{hermitelemma}
The proof of Theorem \ref{main} essentially reduces to
calculating asymptotics of the (appropriately rescaled)
correlation kernel $K_n^{\tau_n}$. The key to this 
analysis is to find a
suitable representation of the Hermite polynomials;
because of the scale factor $\tau^k$ appearing in (\ref{defkernel}) we cannot
use the Christoffel-Darboux formula to simplify the sum,
which is the standard technique. The formula derived in Lemma \ref{hermite}
differs from the one used in \cite{Verbaarschot} for the special 
case $\sigma_n=cn^{-1/3}$, and it enables us to carry out a more 
complete analysis.
\begin{lemma}\label{hermite} 
Let
$r_1, r_2$
and 
$\tau$ 
be positive real numbers such that 
$r_1<\tau r_2$ 
and define the contours 
\begin{equation}
[-\pi,\pi]\ni t \mapsto \gamma_{r_1}(t)=r_1e^{it} \in \mathbb{C}
\end{equation}
and 
\begin{equation}
 \mathbb{R}\ni t \mapsto \Gamma_{r_2}(t)=r_2+i t \in \mathbb{C}
\end{equation}
in the complex plane. Then, for any positive integer $n$ and all 
complex numbers $z_1$ and $z_2$, the identity
\begin{equation}\label{hermiteformula}
\sum_{k=0}^{n-1}\tau^k h_k(z_1)h_k(z_2)
=\frac{\tau^ne^{z_2^2}}{2\pi^2}\oint_{\gamma_{r_1}}\int_{\Gamma_{r_2}}
\left(\frac{w_2}{w_1}\right)^n\frac{e^{w_2^2-2z_2w_2+2z_1w_1-w_1^2}}
{w_1-\tau w_2}\ud w_2 \ud w_1
\end{equation} 
holds.
\end{lemma}
\begin{proof}
Recall that the orthonormal Hermite polynomials can be written 
\begin{equation}
h_k(z)=\pi^{1/4}\sqrt{2^k k!} H_k(z),
\end{equation}
 where
\begin{equation}\label{rodrigues} 
H_k(z)=(-1)^ke^{z^2}\frac{\ud^k(e^{-z^2})}{\ud z^k},
\end{equation}
and that the renormalized polynomials $H_k$ satisfy the generating 
function relation
\begin{equation}\label{generatingfunction}
e^{2wz-w^2}=\sum_{n=1}^{\infty}\frac{H_n(z)w^n}{n!}
\end{equation}
for all complex numbers $z$ and $w$.
Note that for any choice of real $r_2$,
\begin{equation}\label{expref}
e^{-z^2}
=\frac{1}{i\sqrt{\pi}}\int_{\Gamma_{r_2}}e^{w^2-2zw}\ud w.
\end{equation}
Using Equations (\ref{rodrigues}) and (\ref{expref}) 
gives a representation
\begin{equation}\label{hermiteintegral1}
H_k(z)=
\frac{2^ke^{z^2}}{i\sqrt{\pi}}\int_{\Gamma_{r_2}}w^ke^{w^2-2zw}\ud w.
\end{equation}
On the other hand, Equation (\ref{generatingfunction})
 and the residue theorem yield, for any choice of $r_1>0$,
\begin{equation}\label{hermiteintegral2}
\frac{1}{2\pi i}\oint_{\gamma_{r_1}}\frac{e^{2wz-w^2}}{w^{k+1}}\ud w
=\frac{H_k(z)}{k!}.
\end{equation}
Combining the representations (\ref{hermiteintegral1}) 
and (\ref{hermiteintegral2}) gives, for any $\tau>0$
such that $|\tau r_2|>r_1$,
\begin{align*}
\sum_{k=0}^{n-1}\tau^k h_k(z_1)&h_k(z_2)
=\sum_{k=0}^{n-1}\frac{H_k(z_1)H_k(z_2)}{\sqrt{\pi}k!(2/\tau)^k}\nonumber\\
=&-\sum_{k=0}^{n-1}\frac{\tau^k e^{z_2^2}}{2\pi^2}
\oint_{\gamma_{r_1}}\int_{\Gamma_{r_2}}\frac{w_2^k}{w_1^{k+1}}
e^{2w_1z-w_1^2+w^2-2zw}\ud w_2 \ud w_1\nonumber\\
=&-\frac{e^{z_2^2}}{2\pi^2}\oint_{\gamma_{r_1}}\int_{\Gamma_{r_2}}
\sum_{k=0}^{n-1}\left(\frac{\tau w_2}{w_1}\right)^k
\frac{e^{2w_1z-w_1^2+w^2-2zw}}{w_1}\ud w_2 \ud w_1\nonumber\\
=&-\frac{e^{z_2^2}}{2\pi^2}\oint_{\gamma_{r_1}}\int_{\Gamma_{r_2}}
\frac{\left(\frac{\tau w_2}{w_1}\right)^n-1}
{\tau w_2 -w_1}e^{2w_1z-w_1^2+w^2-2zw}\ud w_2 \ud w_1\nonumber\\
=&\frac{\tau^ne^{z_2^2}}{2\pi^2}\oint_{\gamma_{r_1}}\int_{\Gamma_{r_2}}
\left(\frac{w_2}{w_1}\right)^n\frac{e^{w_2^2-2z_2w_2+2z_1w_1-w_1^2}}
{w_1-\tau w_2}\ud w_2 \ud w_1\nonumber\\
\end{align*}
as claimed, where the last equality follows since 
\begin{equation*}
\oint_{\gamma_{r_1}}\frac{e^{2w_1z-w_1^2+w^2-2zw}}{\tau w_2 -w_1}\ud w_1=0,
\end{equation*}
by Cauchy's theorem.
\end{proof}
\subsection{The saddle point argument}
\label{intrepcorrkernel}
Identifying  $\zeta_j=\xi_j+i\eta_j\in \mathbb{C}$ with 
$(\xi_j,\eta_j)\in \mathbb{R}^2$ and using Lemma \ref{hermite}, 
Equation (\ref{defkernel}) can be written
\begin{multline}\label{integralkernel} 
K_n^{\tau}(\zeta_1,\zeta_2)
=\frac{n\tau^n}
{2\pi^{5/2}\sqrt{(1-\tau^2)}}
\exp\left\{-\frac{n}{2}\left(-\frac{\overline{\zeta}_2^2}{\tau}+\frac{\xi_1^2+\xi_2^2}{1+\tau}
+\frac{\eta_1^2+\eta_2^2}{1-\tau}\right)\right\} \\
\times \oint_{\gamma_{r_1}}\int_{\Gamma_{r_2}}
\frac{e^{n\log w_2+w_2^2-
\sqrt{\frac{2n}{\tau}}\overline{\zeta}_2 w_2
-\left(n\log w_1+w_1^2-\sqrt{\frac{2n}{\tau}}\zeta_1w_1\right)}}
{w_1-\tau w_2}\ud w_2 \ud w_1,
\end{multline}
for appropriate choices of $r_1$ and $r_2$, whenever $0<\tau<1$. 
(For the Ginibre ensemble, $\tau=0$, the correlation kernel has 
the form
\begin{equation*}
K_n^{0}\left(\zeta_1,\zeta_2\right)\\
=\frac{n}{\pi}
\sum_{k=0}^{n-1}\frac{\left(\zeta_1 \overline{\zeta}_2\right)^k}{k!}
e^{-\frac{n}{2}\left(|\zeta_1|^2+|\zeta_2|^2\right)},
\end{equation*}
and a representation

\begin{equation*}
\sum_{k=0}^{n-1}\frac{\left(\zeta_1 \overline{\zeta}_2\right)^k}{k!}
=\frac{1}{2\pi i}\int_{\Gamma_{r}}
\frac{e^{\zeta_1\overline{\zeta}_2 w}}{w^n(1-w)}\ud w
\end{equation*}
 analogous to (\ref{hermiteformula}) gives a (simpler) 
saddle point argument, 
paralleling the $\tau>0$ case. The details will be skipped.)

The correlation kernel of the rescaled point process 
$\tilde{Z}_n^{\tau_n}$ is given
by 
\begin{multline} 
\tilde{K}_n^{\tau_n}(\zeta_1,\zeta_2)\\
=a_n b_n \frac{\tau_n}{2n}K_n^{\tau_n}\left(\sqrt{\frac{\tau_n}{2n}}(c_n+a_n\xi_1+ib_n\eta_1),\sqrt{\frac{\tau_n}{2n}}(c_n+a_n\xi_2+ib_n\eta_2)\right),
\end{multline}
 where we will see that the scaling parameters
\begin{equation}\label{defparameters}
(a_n,b_n,c_n):
=\sqrt{\frac{2n}{\tau_n}}(\tilde{a}_n,\tilde{b}_n,\tilde{c}_n),
\end{equation} 
should  be chosen as in the statement of the theorem.

Define the analytic function
\begin{equation}f_n(w)=n \log w+w^2-c_nw, 
\end{equation}where we choose the principal branch of the logarithm.
We expect $\tilde{c}_n$ to be close to
$(1+\tau_n)$, the rightmost edge of the spectrum on 
the global scale, so the main contribution to the exponent 
in the integral in (\ref{integralkernel}) should be
  $f_n(w_2)-f_n(w_1)$.
 The idea of the proof is now to calculate the large 
$n$ asymptotics 
of $\tilde{K}_n^{\tau_n}$ by a saddle point 
argument for $f_n$.
Define the shift parameter
\begin{equation}
\delta_n:=c_n-\sqrt{\frac{2n}{\tau_n}}(1+\tau_n).
\end{equation}
Provided $0\leq x_n=o(1)$, where 
\begin{equation}
x_n=\frac{\delta_n}{\sqrt{n}(1-\tau_n)^2},
\end{equation}
which will turn out always 
to be the case, $f_n$ has two distinct, real, positive saddle points, $w_+>w_-,$ 
solving the equation
\begin{equation}f_n'(w)=\frac{n}{w}+2w-c_n=0.
\end{equation}
Explicitly,
\begin{equation}
w_+=
\frac{c_n}{4}\left(1+\sqrt{1-\frac{8n}{c_n^2}}\right)
=\sqrt{\frac{n}{2\tau_n}}+\frac{\delta_n}{2(1-\tau_n)}
+\mathcal{O}\left(\frac{\delta_n^2}{\sqrt{n}(1-\tau_n)^3}\right), 
\end{equation} 
and
\begin{equation}
w_-=\frac{c_n}{4}\left(1-\sqrt{1-\frac{8n}{c_n^2}}\right)
=\sqrt{\frac{\tau_n n}{2}}-\frac{\tau_n\delta_n}{2(1-\tau_n)}
+\mathcal{O}\left(\frac{\delta_n^2}{\sqrt{n}(1-\tau_n)^3}\right),
\end{equation}
 where the asymptotics are given by an expansion 
of the square root,
\begin{multline}
\label{taylorroot}
\sqrt{1-\frac{8n}{c_n^2}}
=\frac{(1-\tau_n)}{(1+\tau_n)}\left(1+\frac{(2\tau_n)^{3/2}}{(1+\tau_n)}x_n
-\frac{\tau_n^2(3-2\tau_n+3\tau_n^2)}{(1+\tau_n)^2}x_n^2
+\mathcal{O}\left( x_n^3\right)\right).
\end{multline}
To analyze the behaviour of $f_n$ close to the saddle points 
we observe that
\begin{equation}\label{fbisplus} 
f_n''(w_+)=2-\frac{n}{w_+^2}=2(1-\tau_n)+\mathcal{O}\left(\frac{\delta_n}{\sqrt{n}(1-\tau_n)}\right)>0,
\end{equation} 
\begin{equation}\label{fbisminus}  
f_n''(w_-)=
-\frac{2}{\tau_n}(1-\tau_n)
+\mathcal{O}\left(\frac{\delta_n}{\sqrt{n}(1-\tau_n)}\right)<0,
\end{equation} 
and
\begin{equation*} 
f_n^{(k)}(w)=(-1)^{k-1}(k-1)!nw^{-k}, \textrm{ for }k\geq 3.
\end{equation*}
For sequences  $(\tau_n)_{n=1}^{\infty}$  such that $|f_n''(w_-)|$ and
$|f_n''(w_+)|$ become very small, we cannot ignore 
the third order terms in the Taylor expansions of $f_n$ at the saddle points; 
more specifically this happens whenever
\begin{equation*} 
\lim_{n\to \infty}\left|\frac{f_n'''(w_+)^2}{f_n''(w_+)^3}\right|
>0,
\end{equation*}
and similarly at $w_-$, that is, whenever $\sigma_n=\mathcal{O}(1)$.
This results in completely different asymptotic 
behaviour of the integral 
in (\ref{integralkernel}) depending on whether or not 
$\sigma_n=\mathcal{O}(1)$, and the choices of contours and 
parameters in the general parameterization of the integral, 
given in Lemma \ref{parameterization}, will differ in the two cases.

In principle, in view of (\ref{fbisplus}) and (\ref{fbisminus}), 
we would like to choose $r_1=w_-$ and $r_2=w_+$ in 
(\ref{integralkernel}) in order to pick up a 
Gaussian integral at each of the saddle points. However, 
since 
\begin{equation}\label{denominator}
\tau_n w_+ -w_-
=\frac{\tau_n\delta_n}{(1-\tau_n)}
\left(1+\mathcal{O}\left(x_n\right)\right), 
\end{equation} 
the integral in (\ref{integralkernel}) will not converge 
for this choice of $r_1$ and $r_2$
unless  $\delta_n>0$, so if $\delta_n=0$ a 
slight modification will be necessary.

For $r_1$ 
and $r_2$ (depending on $n$) to be specified, parameterize the contours 
of integration 
\begin{equation}
\left\{\begin{array}{ll}\gamma_{r_1}:s\mapsto r_1e^{i\theta_ns}, 
|s|\leq\pi/\theta\\
\Gamma_{r_2}:t\mapsto r_2+i\alpha_nt, t\in \mathbb{R},
\end{array}\right.
\end{equation}
where $\theta_n$ and $\alpha_n$ are positive parameters.
 Moving out the main contributing factor 
from the integral, we get the following representation of the 
correlation kernel:
\begin{lemma}\label{parameterization}
Let  $r_1$ and $r_2$ be any positive real numbers  such that 
$\tau_n r_2-r_1>0$.
For $t,v\in \mathbb{R}$, define
\begin{align}
\Xi_{r_2}^{\zeta_2}(t):=&f_n(r_2+i\alpha_n t)-f_n(r_2)
-i\alpha_n t(a_n \xi_2-ib_n\eta_2)\nonumber\\
=&n\log\left(1+\frac{i\alpha_n t}{r_2}\right)-\alpha_n^2t^2
-i\alpha_n t\left(c_n+a_n\xi_2-ib_n\eta_2-2r_2\right)
\end{align}
and 
\begin{align}\label{phipsi}
\Phi_{r_1}^{\zeta_1}(v)+i&\Psi_{r_1}^{\zeta_1}(v):=-f_n(r_1e^{iv})+f_n(r_1)
+r_1(e^{iv}-1)(a_n \xi_1+ib_n\eta_1)\nonumber\\
=&r_1^2(1-2\cos 2v)+r_1c_n'(\cos v -1)-r_1b_n\eta_1 \sin v\nonumber\\
&+i\left((-n+1)v-r_1^2\sin 2v+r_1c_n'\sin v-r_1b_n\eta_1(\cos v-1)\right)
\end{align}
where $c_n'=c_n+a_n\xi_1$ and $\Phi_{r_1}^{\zeta_1}$ and 
$\Psi_{r_1}^{\zeta_1}$ are the real and imaginary parts of the 
right hand side of (\ref{phipsi}) respectively.

The correlation kernel $\tilde{K}_n^{\tau_n}$ 
of the rescaled determinantal point process $\tilde{Z}_n^{\tau_n}$ can 
then be expressed
\begin{align}\label{rescaledkernel}
\tilde{K}_n^{\tau_n}(\zeta_1,\zeta_2)
=&\frac{-r_1\theta_n\alpha_n a_nb_n\tau_n^{n+1}}{4\pi^{5/2}
\sqrt{(1-\tau_n^2)}}\exp \left\{\frac{(c_n+a_n\xi_2)^2
-\tau_n(c_n+a_n\xi_1)^2}{4(1+\tau_n)}\right.\nonumber\\
&\left. -\frac{b_n^2(\tau_n\eta_1^2+\eta_2^2)}{4(1-\tau_n)}
 -\frac{ib_n\eta_2(c_n+a_n\xi_2)}{2}\right\}\nonumber\\
&\times \exp\left\{f_n(r_2)-f_n(r_1)+a_n(r_1\xi_1 -r_2\xi_2)
+ib_n(r_1\eta_1+r_2\eta_2)\right\}\nonumber\\
&\times\iint_{\{|s|<\pi/\theta\}\times 
\mathbb{R}}F_{r_1,r_2}^{\zeta_1,\zeta_2}(s,t) \ud t \ud s,
\end{align}
where 
\begin{equation}\label{integranddef}
F_{r_1,r_2}^{\zeta_1,\zeta_2}(s,t)
:=\frac{\exp\left\{\Phi_{r_1}^{\zeta_1}(\theta_n s)
+i\Psi_{r_1}^{\zeta_1}(\theta_n s)+\Xi_{r_2}^{\zeta_2}(t)\right\}}{r_1e^{i\theta_n s}-\tau_n (r_2+i\alpha_n t)}.
\end{equation}
\end{lemma}
The technical part of the proof of Theorem \ref{main} has now been reduced
to estimating the integral 
\begin{equation}\label{integralF}
G(\zeta_1,\zeta_2):=\iint_{\{|s|<\pi/\theta\}\times 
\mathbb{R}}F_{r_1,r_2}^{\zeta_1,\zeta_2}(s,t) \ud t \ud s.
\end{equation}
In order to estimate the $s$-integral, 
we need to 
establish that the real function 
\begin{equation*}
\Phi_{r_1}^{\zeta}(v)=r_1^2(1-2\cos 2v)+r_1c_n'(\cos v -1)-r_1b_n\eta_1 \sin v
\end{equation*} decreases monotonically 
as one moves away from its global maximum. This is the content of the 
following lemma.

\begin{lemma}\label{trigfunction} 
Let $\zeta \in \mathbb{R}^2$ be fixed. Put $v_{0}=-b_n\eta/({c_n'}-4r_1)$ and choose 
$r_1=w_-$ if $\sigma_n\to \infty$, and $r_1=w_-(1-\epsilon_n n^{-1/3})$ 
otherwise, where $n^{-1/3}\ll \epsilon_n=o(1)$. 
Then for every sufficiently large $n$ there is a $\rho_n$ with
\begin{equation*}
\left\{\begin{array}{ll}\rho_n <C_1\frac{1}{\epsilon_n n^{1/3}} \textrm{ if }\sigma_n=o(1)\\
\rho_n <C_2\frac{b_n^2}{n(1-\tau_n)^3} \textrm{ otherwise,}\end{array}\right.
\end{equation*} 
such that if $v_1\in (-\pi/2,\pi/2)\setminus \left(v_{0}-\rho_n,v_{0}+\rho_n\right) $, $v_2\in [-\pi,\pi]$ and 
\begin{equation*}
\left|v_2-v_{0}\right|>\left|v_1-v_{0}\right|,
\end{equation*}
then 
\begin{equation*}\Phi_{r_1}^{\zeta}(v_1)>\Phi_{r_1}^{\zeta}(v_2).
\end{equation*}
\end{lemma}

\begin{proof} $\Phi_{r_1}^{\zeta}(v)$ is 
differentiable everywhere, 
so any local extreme points $v$ satisfy ${\Phi_{r_1}^{\zeta}}'(v)=0$, or
 \begin{equation}\label{extremepoints}
\sin v\left(4r_1-\frac{{c_n'}}{\cos v}\right)-b_n\eta
=0,
\end{equation} 
provided $\eta\neq 0$. Put $x=\sqrt{\frac{\tau_n n}{2}}-r_1$. 
Now
\begin{multline}
r_1(c_n'-4r_1)\\
=n(1-\tau_n)+(\delta_n+2x(3-1/\tau_n))\sqrt{\frac{\tau_n n}{2}}
+\mathcal{O}\left(x^2\right)+\mathcal{O}\left(a_n\sqrt{n}\right)>0 
\end{multline} 
if $n$ is sufficiently large, so
\begin{equation*}
\left|4r_1-\frac{{c_n'}}{\cos v}\right|\geq {c_n'}-4r_1
=\mathcal{O}\left(x\right)  \textrm{ if }\sigma_n \to 0,
\end{equation*}
and
\begin{equation*}\left|4r_1-\frac{{c_n'}}{\cos v}\right|\geq {c_n'}-4r_1
=\mathcal{O}\left(\sqrt{n}(1-\tau_n)\right)  \textrm{ otherwise.}
\end{equation*} 
In both cases,
\begin{equation*}
\left|\frac{b_n}{4r_1-{c_n'}/\cos v}\right|=o(1).
\end{equation*}
Therefore, for any solution $v$ to (\ref{extremepoints}), 
$\sin v=\mathcal{O}\left(v_{0}\right)=o(1)$. 
Consider first the case that $|v|=o(1)$. 
Then there are numbers $\epsilon_1$ and $\epsilon_2$ with 
$|\epsilon_1|<|v^3|/6$ and $|\epsilon_2|<v^2$ such that 
(\ref{extremepoints}) can be written 
\begin{equation*}
(v-\epsilon_1)(4r_1-{c_n'}(1+\epsilon_2))-b_n\eta=0,
\end{equation*}
 which implies that 
there is a constant $C$ such that
\begin{equation*}
v\in (v_{0}-\rho_n,v_0+\rho_n),
\end{equation*} 
where
\begin{equation*}
\rho_n<C'\frac{v_0{c_n'}\epsilon_2}{{c_n'}-4r_1}<C\frac{v_0{c_n'}b_n^2\eta^2}{({c_n'}-4r_1^2)^3}.
\end{equation*}
Similarly, if $|v-\pi|= o(1)$ a Taylor expansion of (\ref{extremepoints}) 
gives the 
condition $v\in (\tilde{v}_0-\rho_n,\tilde{v}_0+\rho_n)$
where  $\tilde{v}_0=\pi-b_n\eta/({c_n'}+4r_1)$.
It follows that 
$\Phi_{r_1}^{\zeta}(v)$ 
is strictly decreasing on $[v_0+\rho_n,\tilde{v}_0-\rho_n]$ and strictly 
increasing on $[\tilde{v}_0-2\pi+\rho_n,{v}_0-\rho_n]$.
Clearly, any possible local 
maxima of 
$\Phi_{r_1}^{\zeta}(v)$ near $\tilde{v}_0$ are smaller than, say,   
$\Phi_{r_1}^{\zeta}(\pm \pi/2)$, so evaluating the order of magnitude of $\rho_n$ and remembering that $\Phi_{r_1}^{\zeta}$ is $2\pi$-periodic gives the conclusion. 
\end{proof}

\subsection{Estimates on  $G(\zeta_1,\zeta_2)$ when $\sigma_n$ tends to infinity}
\label{estimatesgumbel}
The two cases, depending on whether or not $\sigma_n=\mathcal{O}(1)$, 
require different choices of the parameters. In this section, estimates 
on $G(\zeta_1,\zeta_2)$ are provided for 
the case $\sigma_n \to \infty $. 
In this regime, the contribution to the exponent in  
$F_{r_1,r_2}^{\zeta_1,\zeta_2}$ from the smaller terms depending on 
the $\eta$-variables is not negligible, so the main contribution 
to $G(\zeta_1,\zeta_2)$ actually comes from intervals $I_n$, $J_n$ 
which in general do not contain the saddle points of $f_n$. 

First, the main contribution
to the integral is estimated in Lemma \ref{localgumbel}, then it is shown 
in Lemma \ref{globalgumbel} that 
the remaining contribution is negligible, and finally in Lemma \ref{integralform} 
the main contribution is explicitly evaluated. 
These estimates 
provide the basis for the proof of part (i) of Theorem \ref{main} in Section \ref{proofgumbel}. 

Throughout this section we will suppose that $\sigma_n \to \infty$ as $n \to \infty$ and that
\begin{equation}\label{deltagumbeldef}
\delta_n
\asymp\sqrt{(1-\tau_n)\log\sigma_n}=n^{-2/3}\sigma_n\sqrt{\log \sigma_n},
\end{equation}
\begin{equation}\label{agumbeldef}
a_n\asymp
\sqrt{\frac{(1-\tau_n)}{\log\sigma_n}}=\frac{n^{-2/3}\sigma_n}{\sqrt{\log \sigma_n}},
\end{equation}
and
\begin{equation}\label{bgumbeldef}
b_n\asymp\left(\frac{n(1-\tau_n)^5}{\log\sigma_n}\right)^{1/4}=\frac{n^{-2/3}\sigma_n^{5/2}}{(\log \sigma_n)^{1/4}}.\end{equation}
Since $\delta_n>0$, 
we may 
choose $r_1=w_-$ and $r_2=w_+$ 
by  
(\ref{denominator}).  
Now choose the parameters $\alpha_n$ and $\theta_n$ so that the exponent in 
$F_{r_1,r_2}^{\zeta_1,\zeta_2}(s,t)$ becomes of order $-s^2/2+t^2/2$, namely
\begin{equation}\label{Alpha}
\alpha_n
=\sqrt{\frac{w_+}{2(w_+-w_-)}}
=\frac{1}{\sqrt{2(1-\tau_n)}}
\left(1-\frac{\tau_n^{3/2}\delta_n}{\sqrt{2n}(1-\tau_n)^2}
+\mathcal{O}\left(x_n^2\right)\right),
\end{equation}
and
\begin{equation}\label{Theta}
\theta_n
=\frac{1}{\sqrt{2w_-(w_+-w_-)}}=\frac{1}{\sqrt{n(1-\tau_n)}}
\left(1-\frac{\tau_n^{3/2}\delta_n}{\sqrt{2n}(1-\tau_n)^2}
+\mathcal{O}\left(x_n^2\right)\right).
\end{equation}
We observe that by definition the parameters are related through the equation 
\begin{equation}\label{alphatheta1}
\alpha_n^2=1/2+\theta_n^2w_-^2
\end{equation} 
and that 
\begin{equation}\label{alphatheta2}
\tau_n\alpha_n^2-\theta_n^2w_-^2
=\frac{\tau_n^{3/2}\delta_n}{\sqrt{2n}(1-\tau_n)^2}
\left(1+\mathcal{O}\left(x_n\right)\right).
\end{equation}

First we estimate the contribution to $G(\zeta_1,\zeta_2)$ 
from near the (slightly shifted) saddle points.  
\begin{lemma}\label{localgumbel}
Put $s_0=-\theta_n w_- b_n \eta_1$ and $t_0= -\alpha_n b_n \eta_2$ 
and let $M_n$, such that $1 \ll M_n \ll \sqrt{\log{\sigma_n}}$, be given. 
Define the sets 
$I_n=\{s:|s-s_0|\leq M_n\}$ 
and 
$J_n=\{t:|t-t_0|\leq M_n\}$ and let 
$\phi_n=n b_n^3(w_+-w_-)^{-3}/{24}$.
Then, for every fixed 
$(\zeta_1,\zeta_2)\in\mathbb{R}^4$, 
\begin{multline*}
\iint_{I_n\times J_n}F_{w_-,w_+}^{\zeta_1,\zeta_2}(s,t)\ud t\ud s\\
=e^{i\phi_n(\eta_2^3-\eta_1^3)}\iint_{I_n\times J_n}
\frac{e^{-\frac{1}{2}t^2-\frac{1}{2}s^2
-i\alpha_n t(a_n \xi_2-ib_n\eta_2)
+iw_- \theta_n s(a_n \xi_1+ib_n\eta_1)}}
{w_- -\tau_n w_+ +iw_- \theta_n s-i\tau_n \alpha_n t}\ud t \ud s\\
+o(1)\frac{ e^{\frac{1}{2}(t_0^2+s_0^2)}}
{|w_- -\tau_n w_+ +iw_- \theta_n s_0-i\tau_n \alpha_n t_0|}.
\end{multline*}

\end{lemma}

\begin{proof}
We expand the exponent in $F_{w_-,w_+}^{\zeta_1,\zeta_2}(s,t)$  and estimate for $s\in I_n$ and $t\in J_n$. Note that, since 
\begin{equation*}
w_-(c_n'-w_-)\theta_n^2
=1+\frac{a_n\xi_1}{(c_n-4w_-)}
=1+\mathcal{O}\left(\frac{1}{\sqrt{(1-\tau_n)n\log\sigma_n}}\right),
\end{equation*} 
 the expansion of the exponent becomes
\begin{align}\label{taylorgumbel}
\Xi_{w_+}^{\zeta_2}(t)+\Phi_{w_-}^{\zeta_1}(\theta_n s)+i&\Psi_{w_-}^{\zeta_1}(\theta_n s)\nonumber\\
=-\frac{1}{2}t^2-i\alpha_n t&(a_n \xi_2-ib_n\eta_2) -\frac{1}{2}s^2
+iw_-\theta_n s(a_n \xi_1+ib_n\eta_1)\nonumber\\
-i\frac{n\alpha_n^3}{3w_+^3}t^3&
+i\frac{\theta_n^3w_-(8w_-- c_n')}{6}s^3-i\frac{w_-\theta_n^2b_n\eta_1}{2}s^2\nonumber\\
&+\mathcal{O}\left(\frac{b_n^4}{n(1-\tau_n)^4}\right)+\mathcal{O}\left(\frac{(1-\tau_n)}{\log \sigma_n}\right).
\end{align}
The imaginary second and third order terms in (\ref{taylorgumbel}) 
are in general not small 
in $I_n\times J_n$, however they are almost constant. For the term involving
$\eta_2$,   
\begin{equation}\label{imaginary2}
\left|i\phi_n\eta_2^3-\left(-\frac{in\alpha_n^3t^3}{3w_+^3}\right)\right|
\leq \frac{n\alpha_n^3}{w_+^3}M_n\max(t_0^2,M_n^2)
\leq C\frac{b_n^2M_n}{\sqrt{n}(1-\tau_n)^{5/2}}.
\end{equation} 
Since
\begin{equation*}w_-(8w_--c_n')=n(3\tau_n-1)
+\mathcal{O}\left(\frac{\sqrt{n}\delta_n}{1-\tau_n}\right)
\end{equation*} 
and
\begin{equation*}\frac{n\theta_n^6w_-^3b_n^3}{3}
=\phi_n\left(1+\mathcal{O}\left(x_n\right)\right),
\end{equation*}
the terms involving  $\eta_1$ can be similarly estimated; 
\begin{multline}\label{imaginary1}
\left|i\phi_n\eta_1^3-\left(i\frac{w_-\theta_n^2b_n\eta_1}{2}s^2
-i\frac{\theta_n^3w_-(8w_--c_n')}{6}s^3\right)\right|\\
=\left|i\phi_n\eta_1^3-in\theta_n^6w_-^3b_n^3\eta_1^3
\left(\frac{(s/s_0)^2}{2n(\theta_n)^2}+\frac{(s/s_0)^3\left(n(3\tau_n-1)(1+
o(1))\right)}{6n}\right)\right| \\
\leq C\frac{b_n^2M_n}{\sqrt{n}(1-\tau_n)^{5/2}}.
\end{multline}
Furthermore, 
\begin{multline}\label{denominatorerror}
w_-e^{i\theta_n s}-\tau_n (w_++i\alpha_n t)\\
=\left(w_-+i\theta_nw_-s-\tau_n (w_++i\alpha_n t)\right)
\left(1+\mathcal{O}\left(\frac{b_n^2}{\sqrt{n}(1-\tau_n)\delta_n}\right)
\right).
\end{multline}
Using Equations (\ref{taylorgumbel}), (\ref{imaginary2}), (\ref{imaginary1}) and 
(\ref{denominatorerror}) and noting that the error terms they give rise to are all small, gives an estimate 
\begin{align*}
&\left|e^{-i\phi_n(\eta_2^3-\eta_1^3)}
\iint_{I_n\times J_n}F_{w_-,w_+}^{\zeta_1,\zeta_2}(s,t)\ud t\ud s\right.
\nonumber\\
&-\left.\iint_{I_n\times J_n}\frac{e^{-\frac{1}{2}t^2 
-\frac{1}{2}s^2-i\alpha_n t(a_n \xi_2-ib_n\eta_2)+
iw_- \theta_n s(a_n \xi_1+ib_n\eta_1)}}
{w_- -\tau_n w_+ +iw_- \theta_n s-i\tau_n \alpha_n t}\ud t \ud s\right|\nonumber\\
\leq & |R_n |\iint_{I_n\times J_n}
\left|\frac{e^{-\frac{1}{2}t^2-\frac{1}{2}s^2-i\alpha_n t(a_n \xi_2-ib_n\eta_2)
+iw_- \theta_n s(a_n \xi_1+ib_n\eta_1)}}
{w_- -\tau_n w_+ +iw_- \theta_n s-i\tau_n \alpha_n t}\right|\ud t \ud s\nonumber\\
=&(1+ o(1))|R_n|\frac{ e^{\frac{1}{2}(t_0^2+s_0^2)}}
{|w_- -\tau_n w_+ +iw_- \theta_n s_0-i\tau_n \alpha_n t_0|},
\end{align*}
where $R_n=o(1)$
and the conclusion follows.
\end{proof}

Next, it is shown that the contributions from outside the set $I_n\times J_n$ 
to $G(\zeta_1,\zeta_2)$ and its 
approximation, given in the previous lemma, are negligible. 
\begin{lemma}\label{globalgumbel}
Define $s_0$, $t_0$, $M_n$, $I_n$, and $J_n$ as in Lemma \ref{localgumbel} 
and let $\theta_n^{-1}\mathbb{T}$ denote the set 
$(-\pi/\theta_n,\pi/\theta_n)$. 
Then, for $(\zeta_1,\zeta_2)\in \mathbb{R}^4$ fixed, 
\begin{multline}
\label{globalgumbelexact}
\left|\iint_{\left(\theta_n^{-1}\mathbb{T}\times\mathbb{R}\right)
\setminus (I_n\times J_n)}
F_{w_-,w_+}^{\zeta_1,\zeta_2}(s,t)\ud t\ud s\right|\\
=o(1)\frac{e^{\frac{1}{2}(t_0^2+s_0^2)}}
{|w_- -\tau_n w_+ +iw_- \theta_n s_0-i\tau_n \alpha_n t_0|}
\end{multline}
and
\begin{multline*}
\left|\iint_{\mathbb{R}^2\setminus(I_n\times J_n)}
\frac{e^{-\frac{1}{2}t^2
-\frac{1}{2}s^2-i\alpha_n t(a_n \xi_2-ib_n\eta_2)
+iw_- \theta_n s(a_n \xi_1+ib_n\eta_1)}}
{w_- -\tau_n w_+ +iw_- \theta_n s-i\tau_n \alpha_n t}\ud t\ud s\right|\\
=o(1)\frac{e^{\frac{1}{2}(t_0^2+s_0^2)}}{|w_- -\tau_n w_+ 
+iw_- \theta_n s_0-i\tau_n \alpha_n t_0|}.
\end{multline*}
\end{lemma}
\begin{proof}
We prove the estimate (\ref{globalgumbelexact}); the second assertion follows by the same argument.
Since 
\begin{equation*} 
(\theta_n^{-1}\mathbb{T}\times\mathbb{R})\setminus (I_n\times J_n)\\
=A_1\cup A_2\cup A_3,
\end{equation*}
where $A_1=I_n\times(\mathbb{R}\setminus J_n)$, $A_2=((\theta_n^{-1}\mathbb{T})\setminus I_n)\times J_n$ and $A_3= ((\theta_n^{-1}\mathbb{T})\setminus I_n)\times(\mathbb{R}\setminus J_n)$,
we consider first the integral 
\begin{align}\label{tgumbel}
I_{J_n^c}(s):
=&\left|\int_{\mathbb{R}\setminus J_n} \frac{e^{\Xi_{w_+}^{\zeta_2}(t)}}{w_-e^{i\theta_n s}-\tau_n (w_++i\alpha_n t)}\ud t\right|\nonumber\\
 \leq&\int_{\mathbb{R}\setminus J_n} 
\frac{\exp\left\{n\log\left|1+\frac{i\alpha_n t}{w_+}\right|
-\alpha_n^2t^2-\alpha_n b_n t\eta_2\right\}}{|w_-+iw_-\theta_n s-\tau_n (w_++i\alpha_n t)|}\ud t\nonumber\\
 \leq& e^{t_0^2/2}\int_{|u|>M_n}\frac{e^{-u^2/2}}{|w_-+iw_-\theta_n s-\tau_n (w_++i\alpha_n (u+t_0))|}\ud u,
\end{align} 
in order to estimate the integral of $F_{w_-,w_+}^{\zeta_1,\zeta_2}$ over $A_1$.
Fix $s\in \mathbb{R}$ and 
suppose first that $\eta_2\neq 0$ and that 
\begin{equation*}
|iw_-\theta_n s-i\tau_n\alpha_n t_0|
\asymp \frac{b_n}{(1-\tau_n)};
\end{equation*} 
this will generically be the case. Note that 
\begin{equation*}
\frac{b_n}{(1-\tau_n)}\gg\tau_nw_+-w_-\asymp \sqrt{\frac{\log \sigma_n}{(1-\tau_n)}}\gg\alpha_n M_n.
\end{equation*} 
We can therefore choose $N_n$ such that 
\begin{equation*}
\alpha_n M_n\ll\alpha_n  N_n 
\ll \left|w_-\theta_n s-\tau_n \alpha_n t_0\right|
\ll\alpha_n N_n^2,
\end{equation*} 
for example 
\begin{equation*}N_n=\frac{\sigma_n}{(\log\sigma_n)^{1/6}}
\end{equation*} 
will do. To see that the contribution to $I_{J_n^c}(s)$ from $|u|>N_n$ is negligible, note that
\begin{align*}
\int_{\pi/\theta_n>|u|>N_n}&\frac{e^{-u^2/2}}
{|w_-+iw_-\theta_n s-\tau_n (w_++i\alpha_n (u+t_0))|}\ud u\nonumber\\
&\times\left(\int_{N_n>|u|>M_n}\frac{e^{-u^2/2}}
{|w_-+iw_-\theta_n s-\tau_n (w_++i\alpha_n (u+t_0))|}
\ud u\right)^{-1}\nonumber\\
 \leq&\frac{\sup_{N_n>|u|>M_n}\left\{|w_-+iw_-\theta_n s
-\tau_n (w_++i\alpha_n (u+t_0))|\right\}}{\tau_nw_+-w_-}\nonumber\\
&\times\frac{\int_{N_n<|u|<L_n}e^{-u^2/2}\ud u}
{\int_{M_n<|u|<N_n}e^{-u^2/2}\ud u}\nonumber\\
\leq & C_1\frac{\alpha_n N_n^2 \sqrt{(1-\tau_n)}}
{\sqrt{\log \sigma_n}}\frac{M_n}{N_n}e^{-N_n^2/3}
 \leq CN_ne^{-N_n^2/3}.
\end{align*}
Equation (\ref{tgumbel}) therefore becomes  
\begin{align*}
I_{J_n^c}(s)
\leq & C_1e^{t_0^2/2}\int_{N_n>|u|>M_n}\frac{e^{-u^2/2}}
{|w_-+iw_-\theta_n s-\tau_n (w_++i\alpha_n (u+t_0))|}\ud u\nonumber\\
 \leq& C\frac{e^{-M_n^2/2}}{M_n}\frac{ e^{t_0^2/2}}
{|w_-+iw_-\theta_n s-\tau_n (w_++i\alpha_n t_0)|}.
\end{align*}
The (simpler)  cases when $\eta_2=0$ or 
$|iw_-\theta_n s-i\tau_n\alpha_n t_0|\ll\ b_n(1-\tau_n)^{-1}$ 
 can be handled similarly, and lead to the same estimate.
Essentially the same argument as in the proof of Lemma \ref{localgumbel}
now provides an estimate
\begin{multline}\label{tgumbel3}
\left|\iint_{A_1}
F_{w_,w_+}^{\zeta_1,\zeta_2}(s,t)\ud t\ud s\right|
=\left|\int_{I_n}I_{J_n^c}(s)e^{\Phi_{w_-}^{\zeta_1}(\theta_n s)+i\Psi_{w_-}^{\zeta_1}(\theta_n s)} \ud s\right|\\
\leq C\frac{e^{-M_n^2/2}}{M_n} e^{t_0^2/2}
\int_{I_n}\frac{e^{\Phi_{w_-}^{\zeta_1}(\theta_n s)}}{|w_- -\tau_n w_+ +iw_- \theta_n s  -i\tau_n \alpha_n t_0|} \ud s\\
\leq C\frac{e^{-M_n^2/2}}{M_n}\frac{e^{t_0^2+s_0^2/2}}{|w_- -\tau_n w_+ +iw_- \theta_n s_0  -i\tau_n \alpha_n t_0|}.
\end{multline}

Now we turn to the integral over $A_2$.
Consider, analogously with the previous estimate, the case that  
$\eta_1\neq 0$ and 
\begin{equation*}
|iw_-\theta_n s_0-i\tau_n\alpha_n t|\gg\tau_nw_+-w_-.
\end{equation*} 
Choose $N_n$ as before and 
let $L_n=(n(1-\tau_n))^{1/5}$, so that 
$L_n \gg N_n$, but $\theta_n L_n\ll 1$. 
Then for $t$ fixed, we proceed to show that the main contribution to 
\begin{align}\label{sgumbel} 
I_{I_n^c}(t):=&\left|\int_{(\theta_n^{-1}\mathbb{T})\setminus I_n}\frac{e^{\Phi_{w_-}^{\zeta_1}(\theta_n s)+i\Psi_{w_-}^{\zeta_1}(\theta_n s)}}{w_-e^{i\theta_n s}-\tau_n (w_++i\alpha_n t)}\ud s\right|
\end{align}
comes from the intervals $M_n \leq |u|\leq N_n$.
First, by Lemma \ref{trigfunction}, 
\begin{align*}
\int_{L_n<|u|<\frac{\pi}{\theta_n}}&
\frac{e^{\Phi_{w_-}^{\zeta_1}(\theta_n (u+s_0))}}
{|w_-e^{i\theta_n (u+s_0)}-\tau_n (w_++i\alpha_n t)|}\ud u\nonumber\\
&\times\left(\int_{M_n<|u|<L_n}
\frac{e^{\Phi_{w_-}^{\zeta_1}(\theta_n (u+s_0))}}
{|w_-e^{i\theta_n (u+s_0)}-\tau_n (w_++i\alpha_n t)|}
\ud u\right)^{-1}\nonumber\\
 \leq&\frac{\sup_{L_n>|u|>M_n}\{|w_-+iw_-\theta_n (u+s_0)
-\tau_n (w_++i\alpha_n t)|\} e^{\Phi_{w_-}^{\zeta_1}(L_n)}}
{\int_{M_n<|u|<L_n} e^{-u^2/2}\ud u} \nonumber\\
& \times \int_{L_n<|u|<\frac{\pi}{\theta_n}}
\frac{\ud u}{\tau_n w_+-w_--w_-\cos(\theta_n(u+s_0))}\nonumber\\
 \leq&C_1\frac{ w_- L_n e^{\frac{1}{2}s_0^2
-\frac{1}{2}L_n^2}}{\sqrt{(\tau_n^2w_+^2-w_-^2)}
} \frac{M_n}{e^{-M_n^2/2}} 
 \leq C(n(1-\tau_n)\log \sigma_n)^{1/4}L_ne^{-L_n^2/3}\ll 1
\end{align*}
But for $|s|<L_n$, we also have an estimate
\begin{align*}
\int_{N_n<|u|<L_n}&\frac{
e^{\Phi_{w_-}^{\zeta_1}(\theta_n (u+s_0))}}
{|w_-e^{i\theta_n (u+s_0)}-\tau_n (w_++i\alpha_n t)|}\ud u\nonumber\\
&\times\left(\int_{M_n<|u|<N_n}
\frac{e^{\Phi_{w_-}^{\zeta_1}(\theta_n (u+s_0))}}
{|w_-e^{i\theta_n (u+s_0)}-\tau_n (w_++i\alpha_n t)|}\ud u\right)^{-1}\nonumber\\
 \leq&\frac{\sup_{M_n<|u|<N_n}\left
\{|w_-+iw_-\theta_n (u+s_0)-\tau_n (w_++i\alpha_n t)|\right\}}
{\tau_nw_+-w_-}\nonumber\\
&\times\frac{\int_{N_n<|u|<L_n}e^{-u^2/2}\ud u}
{\int_{M_n<|u|<N_n}e^{-u^2/2}\ud u}\nonumber\\
\leq & C_1\frac{\alpha_n N_n^2 \sqrt{(1-\tau_n)}}{\sqrt{\log \sigma_n}}\frac{M_n}{N_n}e^{-N_n^2/3}
 \leq CN_ne^{-N_n^2/3}.
\end{align*}
Consequently, (\ref{sgumbel}) becomes
\begin{align*}
I_{I_n^c}(t) \leq &\int_{M_n<|u|<\frac{\pi}{\theta_n}}
\frac{e^{\Phi_{w_-}^{\zeta_1}(\theta_n (u+s_0))}}
{|w_-e^{i\theta_n (u+s_0)}-\tau_n (w_++i\alpha_n t)|}\ud u\nonumber\\
 \leq&C_1\int_{M_n<|u|<N_n}\frac{e^{\Phi_{w_-}^{\zeta_1}(\theta_n s)
}}{|w_-e^{i\theta_n s}
-\tau_n (w_++i\alpha_n t)|}\ud s\nonumber\\
\leq &C \frac{e^{-M_n^2/2}}{M_n}\frac{e^{\frac{1}{2}s_0^2}}
{|w_- -\tau_n w_+ +iw_- \theta_n s_0-i\tau_n \alpha_n t|},
\end{align*} 
so 
\begin{multline}\label{sgumbel3}
\left|\iint_{A_2}
F_{w_,w_+}^{\zeta_1,\zeta_2}(s,t)\ud t\ud s\right|\\
\leq C_1\frac{e^{-M_n^2/2}}{M_n} e^{s_0^2/2}
\int_{J_n}\frac{e^{\Xi_{w_+}^{\zeta_2}(t)}}{|w_- -\tau_n w_+ +iw_- \theta_n s_0  -i\tau_n \alpha_n t|} \ud t \\
\leq C\frac{e^{-M_n^2/2}}{M_n}\frac{e^{t_0^2+s_0^2/2}}{|w_- -\tau_n w_+ +iw_- \theta_n s_0  -i\tau_n \alpha_n t_0|}.
\end{multline}
Clearly,
\begin{multline}\label{sgumbel4}
\left|\iint_{A_3}
F_{w_,w_+}^{\zeta_1,\zeta_2}(s,t)\ud t\ud s\right|
= o(1)\frac{e^{-M_n^2/2}}{M_n}\frac{e^{t_0^2+s_0^2/2}}{|w_- -\tau_n w_+ +iw_- \theta_n s_0  -i\tau_n \alpha_n t_0|},
\end{multline}
which together with (\ref{tgumbel3}) and (\ref{sgumbel3}) 
completes the proof.
\end{proof}

We close this section with an explicit evaluation of the main 
contribution to the approximating integral from Lemma \ref{localgumbel}.
\begin{lemma}\label{integralform}
For any fixed $(\zeta_1,\zeta_2)\in \mathbb{R}^4$, 
\begin{multline*}
\iint_{\mathbb{R}^2}\frac{\exp\left\{-\frac{1}{2}t^2-\frac{1}{2}s^2
-i\alpha_n t(a_n \xi_2-ib_n\eta_2)+iw_- \theta_n s(a_n \xi_1+ib_n\eta_1)\right\}}
{w_- -\tau_n w_+ +iw_- \theta_n s-i\tau_n \alpha_n t}\ud t\ud s\\
=-\frac{2\pi\exp\left\{-\frac{1}{2}\alpha_n^2(a_n\xi_2-ib_n\eta_2)^2
-\frac{1}{2}\theta_n^2w_-^2(a_n\xi_1+i b_n\eta_1)^2\right\}}{\tau_nw_+-w_-
+\tau\alpha_n^2(a_n\xi_2-i b_n\eta_2)+\theta_n^2w_-^2(a_n\xi_1+ib_n\eta_1)}(1+o(1)).
\end{multline*}
\end{lemma}
\begin{proof} Since $\delta_n>0$, we have $\tau_nw_+-w_->0$ and 
the denominator can be written
 \begin{multline*}
\frac{1}{w_- -\tau_n w_+ +iw_- \theta_n s-i\tau_n \alpha_n t}\\
=-\int_0^{\infty}\exp\{-r(\tau_nw_+-w_--iw_- \theta_n s+i\tau_n \alpha_n t\}\ud r,  
\end{multline*}
so, using Fubini's Theorem, 
\begin{multline}\label{mainevaluated}
  \iint_{\mathbb{R}^2}\frac{\exp\left\{-\frac{1}{2}t^2-\frac{1}{2}s^2
-i\alpha_n t(a_n \xi_2-ib_n\eta_2)+iw_- \theta_n s(a_n \xi_1+ib_n\eta_1)\right\}}
{w_- -\tau_n w_+ +iw_- \theta_n s-i\tau_n \alpha_n t}\ud t\ud s \\
= -2\pi \exp\left\{-\frac{\alpha_n^2}{2}(-a_n \xi_2+ib_n\eta_2)^2-\frac{w_-^2\theta_n^2}{2}(a_n \xi_1+ib_n\eta_2)^2\right\}\\
\times\int_0^{\infty}\exp\left\{-A_1 r^2-(B_n+iC_n) r\right\}\ud r, 
\end{multline}
where 
\begin{align*}A_n&=\tau_n^2\alpha_n^2+w_-^2\theta_n^2=\frac{\tau_n(1+\tau_n)}{2(1-\tau_n)}(1+o(1)),\nonumber\\
B_n&=\tau_nw_+-w_-+\tau_n\alpha_n^2a_n \xi_2+w_-^2\theta_n^2a_n \xi_1
=\frac{\tau_n\delta_n}{(1-\tau_n)}(1+o(1)), \textrm{ and }\nonumber\\
C_n&=-\tau_n\alpha_n^2b_n\eta_2+w_-^2\theta_n^2b_n\eta_1.
\end{align*}
Putting $\epsilon=B_n^{1/2}A_n^{-3/4}$ and integrating by parts gives an estimate
\begin{align}\left|\int_0^{\infty}\right.&\left.\exp\left\{-A_n r^2-(B_n+iC_n) r\right\}\ud r
-\frac{1}{B_n+iC_n}\right|\nonumber\\
 &=\left|\frac{A_n}{B_n+iC_n}\int_0^{\infty}2r\exp\left\{-A_n r^2-(B_n+iC_n) r\right\}\ud r\right|\nonumber\\
& \leq\left|\frac{A_n}{B_n+iC_n}\right|\left(\int_0^{\epsilon}2re^{-B_n r}\ud r
+\int_{\epsilon}^{\infty}2re^{-A_n r^2}\ud r\right)\nonumber\\
& \leq\frac{1}{|B_n+iC_n|}\left(\frac{2\epsilon A_n}{B_n}
+e^{-A_n\epsilon^2}\right)\nonumber\\
& \leq C\frac{1}{|B_n+iC_n|}\left(\frac{(1-\tau_n)^{1/4}}{\sqrt{\delta_n}}+e^{-\sqrt{\log\sigma_n}}\right),
\end{align}
which, inserted into (\ref{mainevaluated}), gives the conclusion.
\end{proof}

\subsection{Proof of Theorem \ref{main}, part (i)}\label{proofgumbel}
To be able to apply Lemma \ref{convergenceoflastparticle} to 
$(\tilde{Z}_n^{\tau_n})_{n=1}^{\infty}$, we first calculate the point-wise 
limit of (a kernel equivalent to) $\tilde{K}_n^{\tau_n}$,
using the estimates of Section \ref{estimatesgumbel}, 
and then prove that, for 
any real  $\xi_0$,
there is an integrable function on $(\xi_0,\infty)\times \mathbb{R}$ which
dominates the functions 
$\{\tilde{K}_n^{\tau_n}(\zeta,\zeta)\}_{n=1}^{\infty}$.
\begin{proof}[Proof of Theorem \ref{main} (i)]
 Suppose $\sigma_n \to \infty$ and choose $\alpha_n$ and $\theta_n$ 
as in the previous section, that is, 
\begin{equation*}
\alpha_n=\sqrt{\frac{w_+}{2(w_+-w_-)}}
\end{equation*}
and
\begin{equation*}
\theta_n=\frac{1}{\sqrt{2w_-(w_+-w_-)}}.
\end{equation*}
 The choice of scaling parameters is 
given by requiring that the rescaled kernel $K_n^{\tau_n}$ 
have a non-trivial 
finite limit. To be able to control the error terms we assume 
from the outset 
that the orders of magnitude are correct, namely that the conditions
\begin{equation}\label{orderofdelta}
\delta_n\asymp\sqrt{(1-\tau_n)\log\sigma_n},
\end{equation}
\begin{equation}\label{orderofa}
a_n\asymp\sqrt{\frac{(1-\tau_n)}{\log\sigma_n}}
\end{equation}
and
\begin{equation}\label{orderofb}
b_n\asymp\left(\frac{n(1-\tau_n)^5}{\log\sigma_n}\right)^{1/4}.
\end{equation}
are satisfied.
 Choose $M_n$ such that $1\ll M_n\ll\sqrt{\log\sigma_n}$.
It then follows from the estimates of Lemmas \ref{localgumbel}, \ref{globalgumbel} 
and \ref{integralform} that for any fixed $(\zeta_1,\zeta_2)\in \mathbb{R}^4$,
\begin{multline*}
e^{ -i\phi_n(\eta_2^3-\eta_1^3)}\iint_{(\theta_n^{-1}\mathbb{T})\times\mathbb{R}}F_{w_-,w_+}^{\zeta_1,\zeta_2}(s,t)\ud t\ud s\\
=-\frac{2\pi \exp\left\{-\frac{1}{2}\alpha_n^2(a_n\xi_2-ib_n\eta_2)^2
-\frac{1}{2}\theta_n^2w_-^2(a_n\xi_1+i b_n\eta_1)^2
\right\}}{\tau_nw_+-w_-+\tau\alpha_n^2(a_n\xi_2-i b_n\eta_2)
+\theta_n^2w_-^2(a_n\xi_1+ib_n\eta_1)}(1+o(1)).
\end{multline*}
Since $\tau_n w_+-w_->0$ whenever $\delta_n>0$, we may choose $r_1=w_-$
and $r_2=w_+$ in the representation (\ref{rescaledkernel}) derived for
the correlation kernel of $\tilde{Z}_n^{\tau_n}$. Recalling (\ref{alphatheta1}) and (\ref{alphatheta2}) 
this yields, after some simplification,
\begin{multline}\label{ksimplified}
\tilde{K}_n^{\tau_n}(\zeta_1,\zeta_2)\\
=\frac{a_nb_n\tau_nw_- \theta_n\alpha_n
\exp \left\{n\log{\tau_n}
+\frac{c_n^2(1-\tau_n)}{4(1+\tau_n)}
+f_n(w_+)-f_n(w_-)
\right\}}
{2\pi^{3/2}\sqrt{(1-\tau_n^2)}
\left(\tau_nw_+-w_-+\tau\alpha_n^2(a_n\xi_2-i b_n\eta_2)
+\theta_n^2w_-^2(a_n\xi_1+ib_n\eta_1)\right)}\\
\times\exp\left\{ -\frac{\tau_n a_n\delta_n}{(1-\tau_n^2)}(\xi_1+\xi_2)(1+o(1))
-\frac{\tau_n^{3/2}b_n^2\delta_n}{2^{3/2}
\sqrt{n}(1-\tau_n)^3}(\eta_1^2+\eta_2^2)(1+o(1))\right\}\\
\times \frac{F_n(\zeta_2)}{F_n(\zeta_1)} \exp\left\{a_n^2\left(\frac{1}{4(1+\tau_n)}-\frac{\alpha_n^2}{2}\right)(\xi_2^2+\xi_1^2)\right\}(1+o(1)),
\end{multline}
where  
$F_n(\zeta)=\exp\{-ib_nw_-\eta+ib_n\theta_n^2w_-^2\xi\eta+i\phi_n\eta^3\}$ 
and the small $o$ terms in the exponent are both 
$\mathcal{O}\left(\delta_nn^{-1/2}(1-\tau_n)^{-1}\right)$.
The correlation kernel of a determinantal point process is not uniquely defined; 
clearly all correlation functions remain unchanged if 
$\tilde{K}_n^{\tau_n}(\zeta_1,\zeta_2)$ is replaced by 
\begin{equation}\label{kprime}
K_n^{'}(\zeta_1,\zeta_2)
=\frac{F_n(\zeta_1)}{F_n(\zeta_2)}\tilde{K}_n^{\tau_n}(\zeta_1,\zeta_2).
\end{equation} 
 In order for the exponent in (\ref{ksimplified}) to have a non-trivial 
finite limit depending on 
the variables $\zeta_j$, we choose $a_n$ and $b_n$ such that
\begin{equation}\label{choicea}
\lim_{n\to \infty}\frac{\tau_n a_n\delta_n}{(1-\tau_n^2)}=
\lim_{n\to \infty}\frac{\tau_n^{3/2}b_n^2\delta_n}{2^{3/2}\sqrt{n}(1-\tau_n)^3}= \frac{1}{2}
\end{equation}
as $n\to \infty$. 
 To evaluate the constant exponential factor in (\ref{ksimplified}) 
the precise asymptotics (\ref{taylorroot}) are needed.   
 Expansions of the logarithmic terms, and a considerable amount of 
subsequent algebraic manipulation, give
\begin{align*}
n\log{\tau_n}&
+\frac{c_n^2(1-\tau_n)}{4(1+\tau_n)}
+f_n(w_+)-f_n(w_-)\nonumber\\
=&\frac{c_n^2(1-\tau_n)}{4(1+\tau_n)}+w_+^2-w_-^2-c_n(w_+-w_-)+
\frac{\sqrt{2\tau_n n}\delta_n}{(1-\tau_n)}\nonumber\\
&+\delta_n^2\left(\frac{\tau_n(-3-2\tau_n-2\tau_n^2-2\tau_n^3
+\tau_n^4)}{4(1+\tau_n)(1-\tau_n)^3}\right)+\mathcal{O}\left(\frac{\delta_n^3}{\sqrt{n}(1-\tau_n)^5}\right)\nonumber\\
 =&-\frac{\tau_n\delta_n^2}{(1-\tau_n^2)}+\mathcal{O}\left(\frac{\delta_n^3}{\sqrt{n}(1-\tau_n)^5}\right).
\end{align*}
After an asymptotic expansion of the denominator of the 
constant factor in (\ref{ksimplified}), Equation (\ref{kprime}) thus becomes
\begin{multline}\label{kprime2}
K_n^{'}(\zeta_1,\zeta_2)\\
=(1+o(1))\frac{a_nb_nw_- \theta_n\alpha_n\sqrt{(1-\tau_n)}
\exp\left\{-\tau_n\delta_n^2(1-\tau_n^2)^{-1}\right\}}
{2\pi^{3/2}\sqrt{(1+\tau_n)}\delta_n\left(1+i\frac{b_n}{2\delta_n}(\eta_1-\eta_2)\right)}e^{-\frac{1}{2}(\xi_1+\xi_2+\eta_1^2+\eta_2^2)}.
\end{multline} 
The exact choice of $\delta_n$ is now given by requiring that  $K_n^{'}(\zeta_1,\zeta_2)$ have a finite limit, say
\begin{equation*}
\lim_{n\to \infty}\frac{a_nb_nw_- \theta_n\alpha_n\sqrt{(1-\tau_n)}
\exp\left\{-\tau_n\delta_n^2(1-\tau_n^2)^{-1}\right\}}{2\pi^{3/2}\sqrt{(1+\tau_n)}\delta_n}
=\frac{1}{{\sqrt{\pi}}},
\end{equation*} 
or equivalently, in view of (\ref{choicea}), 
\begin{equation}\label{deltaequation}
\lim_{n\to \infty}\frac{\sqrt{(1+\tau_n)}}{2^{11/4}
\tau_n^{5/4}\pi}n^{1/4}(1-\tau_n)^2\delta_n^{-5/2}
\exp\left\{\frac{-\tau_n\delta_n^2}{(1-\tau_n^2)}\right\}= 1.
\end{equation}
Equation (\ref{deltaequation}) is satisfied for
\begin{multline*}\delta_n= \sqrt{\frac{(1+\tau_n)}{4\tau_n}}
\sqrt{(1-\tau_n)6\log\sigma_n}\\
-\sqrt{\frac{(1+\tau_n)}{\tau_n}}
\sqrt{\frac{(1-\tau_n)}{6\log\sigma_n}}
\left(\frac{5}{4}\log(6\log\sigma_n)
+\log(2^{1/4}(1+\tau_n)^{3/4}\pi)\right), 
\end{multline*}
so by (\ref{choicea}) we may choose 
\begin{equation*}a_n
=\sqrt{\frac{(1+\tau_n)}{\tau_n}}
\sqrt{\frac{(1-\tau_n)}{6\log\sigma_n}},
\end{equation*}
and
\begin{equation*}b_n
=\left(\frac{8}{\tau_n^2(1+\tau_n)}\right)^{1/4}
\left(\frac{n(1-\tau_n)^5}
{6\log\sigma_n}\right)^{1/4},
\end{equation*}
which are the choices in the statement of the theorem.
Since these parameters satisfy the assumptions (\ref{orderofdelta}) 
through (\ref{orderofb}) and $b_n \gg\delta_n$, 
\begin{equation*}\label{almostext}
K_n^{'}(\zeta_1,\zeta_2)
=\frac{e^{-\frac{1}{2}(\xi_1+\xi_2)
-\frac{1}{2}(\eta_1^2+\eta_2^2)}}
{\sqrt{\pi}\left(1+i\frac{b_n}{2\delta_n}(\eta_1-\eta_2)\right)}
(1+o(1))\to M_{P2}(\zeta_1,\zeta_2), 
\end{equation*}
as  $n\to \infty$.

It remains to prove that, for any given 
$\xi_0\in \mathbb{R}$,  
$\tilde{K}_n^{\tau_n}(\zeta,\zeta)$ 
is dominated by an integrable function on
 $(\xi_0,\infty)\times \mathbb{R}$ for every sufficiently large $n$.
To estimate the integral 
\begin{equation*} 
\iint_{(\theta_n^{-1}\mathbb{T})\times\mathbb{R}}
F_{w_-,w_+}^{\zeta,\zeta}(s,t)\ud t\ud s
\end{equation*} 
from above, let $\epsilon<1$ be fixed and consider any $n$ so large that
\begin{equation}
|\xi_0|<\epsilon\frac{(1-\tau_{n})(c_n-4w_-)}{(1+\tau_{n})b_n^2 a_{n}}\asymp \frac{\log \sigma_{n}}{(1-\tau_{n})}.
\end{equation}

For such a choice of $n$, it follows that
\begin{equation*}\epsilon_1:=\frac{3\alpha_n^2a_n|\xi_0|}{w_+}
<\epsilon\frac{4(1-\tau_n)}{(1+\tau_n)b_n^2}\asymp\sigma_n^{-3}\sqrt{\log\sigma_n},
\end{equation*}
\begin{equation*}
\epsilon_2:
=\frac{|\xi_0|a_n}{c_n-4w_-} 
<\epsilon\frac{(1-\tau_n)}{(1+\tau_n)b_n^2},
\end{equation*} 
and  that 
\begin{equation*}
\frac{\tau_na_n|\xi_0|}{(\tau_n w_+-w_-)}
<
\epsilon\frac{\tau_n}{(1+\tau_n)}\left(1+\mathcal{O}\left(\frac{1}{\log\sigma_n}\right)\right)<\epsilon.
\end{equation*}
To capture some of the oscillatory terms in the exponent of 
$F_{w_-,w_+}^{\zeta,\zeta}$, 
we may again change 
contours of integration, by Cauchy's theorem. Replacing the real 
line by the contour $\mathbb{R}-ia_n\xi/(2\alpha_n)$ 
in the $t$-integral, 
and putting $t'=t+ia_n\xi/(2\alpha_n)$, gives 
\begin{multline}\label{factorized1}
 \left|\iint_{(\theta_n^{-1}\mathbb{T})\times\mathbb{R}}
F_{w_-,w_+}^{\zeta,\zeta}(s,t)\ud t\ud s\right|\\
\leq\frac{1}{(\tau_n w_+-w_--\tau_n a_n|\xi_0|)}\int_{(\theta_n^{-1}\mathbb{T})}
e^{\Phi_{w_-}^{\zeta}(\theta_n s)}\ud s\left|\int_{\mathbb{R}}
e^{\Xi_{w_+}^{\zeta}\left(t'-ia_n\xi/{(2\alpha_n)}\right)} \ud t'\right|.
\end{multline}
Now, using that $f_n'(w_+)=0$ and the inequality 
$\log(1+x)\leq x$ gives an estimate
\begin{align}\label{dominated1}
\left|\int_{\mathbb{R}}
\exp\right.&\left.\left\{\Xi_{w_+}^{\zeta}\left(t'-ia_n\xi/{(2\alpha_n)}\right)\right\} \ud t'\right|\nonumber\\
 \leq&\int_{\mathbb{R}}
\exp\left\{\operatorname{Re}\left(\Xi_{w_+}^{\zeta}\left(t'-ia_n\xi/{(2\alpha_n)}\right)\right)\right\}\ud t'\nonumber\\
 \leq&\int_{\mathbb{R}}\exp\left\{-\frac{1}{2}t'^2\left(2\alpha_n^2-\frac{n\alpha_n^2}{(w_++a_n \xi/2)^2}\right)+\alpha_nb_n\eta  t'
\right.\nonumber\\
&\left.+n\log\left(1+\frac{a_n \xi}{2w_+}\right)+a_n\xi(w_+-c_n/2)-\frac{a_n^2\xi^2}{4}\right\}\nonumber\\
 \leq&\int_{\mathbb{R}}\exp\left\{-\frac{1}{2}t'^2\left(1-\epsilon_1\right)+\alpha_nb_n\eta t'-\frac{a_n^2\xi^2}{4}\right\}\nonumber\\
 =&\sqrt{\frac{2\pi}{\left(1-\epsilon_1\right)}}\exp\left\{\frac{\alpha_n^2b_n^2\eta^2}{2\left(1-\epsilon_1\right)}-\frac{a_n^2\xi^2}{4}\right\}.
\end{align}
To obtain a uniform bound in $n$ of the first integral, recall that 
\begin{equation*}
c_n'-4w_-\geq (c_n-4w_-)\left(1-\frac{|\xi_0|a_n}{c_n-4w_-}\right)
\end{equation*}
where  $c_n'=c_n+a_n\xi$.
Suppose without loss of generality that $\eta \leq 0$. 
Then there is a 
$v_{\eta}\in [0,\pi/2)$ such that  ${\Phi_{w_-}^{\zeta}}'(v_{\eta})=0$ 
and since 
\begin{align}
{\Phi_{w_-}^{\zeta}}''(v_{\eta})
&=4w_-^2\cos 2v_{\eta}-w_-c_n'\cos v_{\eta}+w_-b_n\eta\sin v_{\eta}\nonumber\\
&<w_-(4w_--c_n)<0,
\end{align}
a saddle point argument gives 
\begin{multline}\label{ssadel}
\int_{|\theta_n s|<\pi}e^{\Phi_{w_-}^{\zeta}(\theta_n s)}\ud s\\
 \leq C_1e^{\Phi_{w_-}^{\zeta}(v_{\eta})}\int_{\mathbb{R}}
\exp\left\{-\frac{1}{2}\left(\frac{c_n'-4w_-}{c_n-4w_-}\right)\left(s-\frac{v_{\eta}}{\theta_n}\right)^2\right\} \ud s
 \leq Ce^{\Phi_{w_-}^{\zeta}(v_{\eta})}.
\end{multline}
Consider the difference
\begin{multline*}
\frac{b_n^2 \eta^2w_-^2\theta_n^2}
{2(1-\epsilon_2)}-\Phi_{w_-}^{\zeta}(v_{\eta})\\
=\frac{b_n^2 \eta^2w_-^2\theta_n^2}
{2(1-\epsilon_2)}+b_nw_-\eta \sin v_{\eta}
-w_-c_n'(\cos v_{\eta}-1)-w_-^2(1-\cos2 v_{\eta}))\\
\geq -\frac{(1-\epsilon_2)\sin^2v_{\eta}}{2\theta_n^2}-c_n'w_-(\cos v_{\eta}-1)-w_-^2(1-\cos2 v_{\eta})=:g(v_{\eta}).
\end{multline*}
Since 
\begin{equation*}\label{gderivativegumbel}
 g'(v_{\eta})=
c_n w_-\sin v_{\eta}\left(\frac{c_n'}{c_n}-(1-\epsilon_2)\cos v_{\eta}\right)\geq 0,
\end{equation*}
$g$ is increasing on $[0,\pi/2)$ so,  noting that $g(0)=0$, it follows 
from (\ref{ssadel}) that
\begin{equation*}
\int_{|\theta_n s|<\pi}e^{\Phi_{w_-}^{\zeta}(\theta_n s)}\ud s
\leq C\exp\left\{\frac{b_n^2 \eta^2w_-^2\theta_n^2}{2(1-\epsilon_2)}\right\}.
\end{equation*}
This, together with the estimates (\ref{factorized1}) and (\ref{dominated1}), inserted into (\ref{rescaledkernel}) shows that  
\begin{align}\label{integrableB}
|\tilde{K}_n^{\tau_n}(\zeta,\zeta)|
 \leq&C_1\exp\left\{\left(\frac{(1-\tau_n)c_n a_n}{2(1+\tau_n)}
-a_n(w_+-w_-)\right)\xi-\frac{\tau_na_n^2\xi^2}{2(1+\tau_n)}\right.\nonumber\\
&\left.-\frac{b_n^2\eta^2}{2}\left(\frac{(1+\tau_n)}{2(1-\tau_n)}
-\frac{\alpha_n^2}{\left(1-\epsilon_1\right)}-\frac{w_-^2\theta_n^2}
{\left(1-\epsilon_2\right)}\right)\right\}\nonumber\\
\leq & C_{\xi_0}e^{-(1-\epsilon)\xi}
\exp\left\{-\eta^2+\frac{b_n^2\eta^2}{2}\left(\frac{\alpha_n^2\epsilon_1}
{\left(1-\epsilon_1\right)}
+\frac{w_-^2\theta_n^2\epsilon_2}{(1-\epsilon_2)}\right)\right\}\nonumber\\
\leq & C_{\xi_0}e^{-(1-\epsilon)\xi-
(1-\epsilon)\eta^2},
\end{align}
where $C_{\xi_0}$ denotes a constant depending on $\xi_0$. This provides an 
integrable bound 
on $\tilde{K}_n^{\tau_n}(\zeta,\zeta)$, so  
by 
Lemma (\ref{convergenceoflastparticle}) and the point-wise convergence 
of $K_n'$ to $M_{P2}$, 
$\tilde{Z}_n^{\tau_n}$ converges weakly to $Z_{P}$ 
and the last particle distribution $F_n^{\tau_n}$ converges to $F_{G}$. 
\end{proof}

\subsection{Estimates on  $G(\zeta_1,\zeta_2)$ when $\sigma_n$ tends to a finite limit}\label{estimatesairy}
Throughout this section, we will suppose that 
$\sigma_n\to \sigma \in [0,\infty)$. 
Again, we prove a series of estimates on $G(\zeta_1,\zeta_2)$ for this case. 
First, in Lemma \ref{localairy}, on 
the main contribution from close 
to the saddle points, and then in Lemmas \ref{globalairy} and 
\ref{globalairy2} on the remaining, small contributions. 
 
Fix the choices $\delta_n=0$,
$a_n=\sqrt{2\tau_n}n^{-1/6}$  and 
$b_n=\sqrt{2\tau_n(1-\tau_n)}$ of the scaling parameters.
Since $\delta_n=0$, $w_--\tau_n w_+=0$ and we must choose contours 
slightly removed 
from the saddle points in order for $G(\zeta_1,\zeta_2)$
to converge, 
say   $r_1=w_--\tau_n\alpha_n\epsilon_n$ and 
$r_2=w_++\alpha_n\epsilon_n$, where 
$\epsilon_n \ll 1 $ will be specified. 
The quadratic terms in the expansions of $f_n$ at the saddle points 
$w_-$ and $w_+$ may now be arbitrarily small (depending on how small 
$\sigma_n$ becomes), so we choose  
$\alpha_n=(2\tau_n)^{-1/2}n^{1/6}$ and $\theta_n=n^{-1/3}$  so 
that the \emph{third} order terms become of order one.

To simplify the calculations by using the fact that $w_+$ 
is a saddle point of $f_n$, we note that, by definition, 
\begin{equation}\label{Falternative}
F_{r_1,r_2}^{\zeta_1,\zeta_2}(s,t)
=F_{r_1,w_+}^{\zeta_1,\zeta_2}(s,t-i\epsilon_n)e^{-f_n(r_2)+f_n(w_+)
+\alpha_n\epsilon_n(a_n\xi_2-ib_n\eta_2)}.
\end{equation}
 
We begin by approximating the main contribution, from close to 
the saddle points.
\begin{lemma}\label{localairy}
Given $\epsilon_n$ and $T_n$, choose 
$r_1=w_--\tau_n\alpha_n\epsilon_n$ and 
$r_2=w_++\alpha_n\epsilon_n$ and  define the contour 
\begin{align*}\tilde{\gamma}_n& :(-T_n,T_n)\to \mathbb{C},\\
 \tilde{\gamma}_n&(t) = t+i\epsilon_n.
\end{align*} 
Let $(\zeta_1,\zeta_2)\in \mathbb{R}^4$ be fixed.

\begin{itemize}

\item[(i)]
If $\sigma>0$ or  $\sigma_n=\mathcal{O} (n^{-2/{15}})$ 
choose  $T_n=n^{k_1}$ 
and $\epsilon_n=n^{-k_2}$ for some
$0<k_1<1/{15}$ 
and  $k_1<k_2<2k_1$.
\item[(ii)]
If  $\sigma=0$ but $\sigma_n n^{2/{15}} \to\infty$ , choose  $T_n=\sigma_n^{-m_1}$ and $\epsilon_n=\sigma_n^{m_2}$, 
for some $0<m_1<1/2$, and $m_1<m_2<2m_1$. 
\end{itemize}
For these choices,
\begin{multline}
\sqrt{\frac{\tau_n}{2}}n^{1/6}e^{-\frac{1}{2}(\eta_1^2+\eta_2^2)}
\iint_{(-T_n,T_n)^2}F_{r_1,w_+}^{\zeta_1,\zeta_2}(s,t-i\epsilon_n)\ud t\ud s\\
=-\int_{\tilde{\gamma}_n}\int_{\tilde{\gamma}_n}
\frac{e^{-\frac{1}{2}(\sigma_n v-\eta_2)^2+\frac{i}{3}v^3
+i\xi_2v -\frac{1}{2}(\sigma_n u+\eta_1)^2+\frac{i}{3}u^3
+i\xi_1u}}{i(u+v)}\ud u\ud v+ o(1).
\end{multline}
\end{lemma}
\begin{remark}The case in which $\sigma=0$ but  
condition (ii) is not
satisfied causes no problem, but for simplicity we omit the details. 
\end{remark} 

\begin{proof}Note that in both cases, $1 \ll T_n \ll n^{1/15}$ and $T_n^{-2}\ll \epsilon_n \ll T_n^{-1}$. 
Put $\tilde{t}=t-i\epsilon_n$ and $\tilde{s}=s+i\epsilon_n$.
Using the expansion 
\begin{align*}\Phi_{r_1}^{\zeta_1}(v)+i\Psi_{r_1}^{\zeta_1}(v)
 =&v(-r_1b_n\eta_1+i(1-n-2r_1^2+r_1c_n'))\nonumber\\
&-\frac{1}{2}(r_1c_n'-4r_1^2)v^2-i\frac{r_1b_n\eta_1v^2}{2}\nonumber\\
&+i\frac{(8r_1^2-r_1c_n')}{6}v^3+\frac{r_1b_n\eta_1}{6}v^3+\mathcal{O}\left(nv^4\right)
\end{align*}
with $r_1=w_--\tau_n\alpha_n\epsilon_n$, and keeping in mind that $\epsilon_n\gg n^{-2/{15}}$, one obtains, after some messy but straight-forward calculations,

\begin{multline*}
\left|\Phi_{r_1}^{\zeta_1}(\theta_ns)+i\Psi_{r_1}^{\zeta_1}(\theta_n s)
-\left(i\tilde{s}(\xi_1+i\sigma_n \eta_1)
-\frac{\sigma_n^2}{2}\tilde{s}^2
+\frac{i}{3}\tilde{s}^3\right)\right|\\
\leq C\left(\epsilon_n+n^{-1/3}s^4\right).
\end{multline*}
 Using that $f_n'(w_+)=0$, the expansion of the exponent in $F_{r_1,w_+}^{\zeta_1,\zeta_2}(s,t-i\epsilon_n)$ becomes
\begin{multline}\label{taylorairy}
\Xi_{w_+}^{\zeta_2}\left(\tilde{t}\right)+\Phi_{r_1}^{\zeta_1}(\theta_n s)+i\Psi_{r_1}^{\zeta_1}(\theta_n s)\\
=-\frac{\sigma_n^2}{2}\tilde{t}^2-\frac{i}{3}\tilde{t}^3-i\tilde{t}(\xi_2-i\sigma_n\eta_2)-\frac{\sigma_n^2}{2}\tilde{s}^2+\frac{i}{3}\tilde{s}^3+i\tilde{s}(\xi_1+i\sigma_n\eta_1)\\
+n^{-1/3}\mathcal{O}\left(t^4+s^4\right)+\mathcal{O}(\epsilon_n).
\end{multline}
Note also that
\begin{equation*}
 \iint_{(-T_n,T_n)^2}\frac{\ud s \ud t}{|i(s-t)-2\epsilon_n|}
\leq CT_n\log{\frac{T_n}{\epsilon_n}}.
\end{equation*}
Consider first the case that either $\sigma>0$, or  
$\sigma_n=\mathcal{O} (n^{-2/{15}})$.
Then, by (\ref{rescaledkernel}) and (\ref{taylorairy}),
\begin{align}\label{localairyestimate1}
\left|\sqrt{\frac{\tau_n}{2}}\right.&\left.n^{1/6}
e^{-\frac{1}{2}(\eta_1^2+\eta_2^2)}
\iint_{(-T_n,T_n)^2}F_{r_1,w_+}^{\zeta_1,\zeta_2}(s,t-i\epsilon_n)\ud t  \ud s
\right.\nonumber\\
&\left.-\iint_{(-T_n,T_n)^2}
\frac{e^{-\frac{1}{2}(\sigma_n\tilde{t}+\eta_2)^2
-\frac{i}{3}\tilde{t}^3-i\tilde{t}\xi_2
 -\frac{1}{2}(\sigma_n\tilde{s}+\eta_1)^2
+\frac{i}{3}\tilde{s}^3+i\tilde{s}\xi_1}}
{i(s-t)-2\epsilon_n}\ud t\ud s\right|\nonumber\\
 \leq&C_1\left(n^{-1/3}T_n^4+\epsilon_n\right)\nonumber\\
&\times\int_{(-T_n,T_n)^2}
\left|\frac{e^{ -\frac{1}{2}(\sigma_n\tilde{t}+\eta_2)^2
-\frac{i}{3}\tilde{t}^3-i\tilde{t}\xi_2 
-\frac{1}{2}(\sigma_n\tilde{s}+\eta_1)^2
+\frac{i}{3}\tilde{s}^3+i\tilde{s}\xi_1}}
{i(s-t)-2\epsilon_n}\right|\ud t\ud s\nonumber\\
 \leq&C_2\left(n^{-1/3}T_n^4+\epsilon_n\right)\int_{(-T_n,T_n)^2}\frac{e^{\epsilon_n(\xi_1-\xi_2) +\sigma_n^2\epsilon_n^2}}
{|i(s-t)-2\epsilon_n|}\ud t\ud s\nonumber\\
 \leq&C_3\log{\frac{T_n}{\epsilon_n}}\left(T_n^5 n^{-1/3}+T_n\epsilon_n\right)
\nonumber\\
 \leq&C\log{n}\left(n^{5k_1-1/3}+n^{k_1-k_2}\right)
\end{align}
Next, suppose $ n^{-2/{15}}\ll\sigma_n = o(1)$. 
Then
\begin{align}\label{localairyestimate2}&\left|\sqrt{\frac{\tau_n}{2}}n^{1/6}
\iint_{(-T_n,T_n)^2}F_{r_1,w_+}^{\zeta_1,\zeta_2}(s,t-i\epsilon_n)\ud s \ud t\right.\nonumber\\
&\left. -\iint_{(-T_n,T_n)^2}
\frac{e^{-\frac{i}{3}\tilde{t}^3-i\tilde{t}\xi_2
+\frac{i}{3}\tilde{s}^3+i\tilde{s}\xi_1}}
{i(s-t)-2\epsilon_n}\ud t\ud s\right|\nonumber\\
 \leq&C_1\left(\sigma_n T_n+n^{-1/3}T_n^4
+\epsilon_n\right)\nonumber\\
&\times\int_{(-T_n,T_n)^2}
\left|\frac{e^{ -\frac{1}{2}(\sigma\tilde{t}+\eta_2)^2
-\frac{i}{3}\tilde{t}^3-i\tilde{t}\xi_2 
-\frac{1}{2}(\sigma\tilde{s}+\eta_1)^2
+\frac{i}{3}\tilde{s}^3+i\tilde{s}\xi_1}}
{i(s-t)-2\epsilon_n}\right|\ud t\ud s\nonumber\\
 \leq&C_2\left(\sigma_n T_n+n^{-1/3}T_n^4
+\epsilon_n\right)\int_{(-T_n,T_n)^2}\frac{e^{\epsilon_n(\xi_1-\xi_2)
}}
{|i(s-t)-2\epsilon_n|}\ud t\ud s
\nonumber\\
 \leq&C_3\left(\sigma_n T_n^2+T_n^5 n^{-1/3}+T_n\epsilon_n\right)
\log{\frac{T_n}{\epsilon_n}}\nonumber\\
\leq& C\log{\sigma_n}\left(\sigma_n^{1-2m_1}+\sigma_n^{-5m_1}n^{-1/3}
+\sigma_n^{-m_1+m_2}\right).
\end{align}
Letting $u=s+i\epsilon_n$ and $v=-(t-i\epsilon_n)$, (\ref{localairyestimate1}) and (\ref{localairyestimate2}) give the conclusion.
\end{proof}

We turn to an estimate of the contribution to the integral of $F_{r_1,w_+}^{\zeta_1,\zeta_2}(s,t-i\epsilon_n)$ from outside 
the set $(-T_n,T_n)^2$.
\begin{lemma}\label{globalairy}
Let $T_n$ and $\epsilon_n$ be as in Lemma \ref{localairy} 
and put $r_1=w_--\tau_n\alpha_n\epsilon_n$. 
Let $(\zeta_1,\zeta_2)\in \mathbb{R}^4$ be fixed. Then
\begin{align}\label{globalairylemma}
\iint_{((\theta_n^{-1}\mathbb{T})\times\mathbb{R})\setminus(-T_n,T_n)^2}
F_{r_1,w_+}^{\zeta_1,\zeta_2}(s,t-i\epsilon_n)\ud t\ud s
= n^{-1/6}o(1).
\end{align}
\end{lemma}
\begin{proof} 

Note that
\begin{align}\label{factorizedF}
\left|
F_{r_1,w_+}^{\zeta_1,\zeta_2}(s,t-i\epsilon_n)\right|
 =& \left|
\frac{\exp\left\{\Xi_{w_+}^{\zeta_2}(\tilde{t})+\Phi_{r_1}^{\zeta_1}(\theta_n s)
+i\Psi_{r_1}^{\zeta_1}(\theta_n s)\right\}}
{(w_- -\tau_n\alpha_n \epsilon_n)
 e^{i\theta_n s}-\tau_n (w_++i\alpha_n\tilde{t})}\right|\nonumber\\
 \leq&\frac{1}{\alpha_n\epsilon_n}e^{\operatorname{Re}\left(\Xi_{\epsilon_n}^{\zeta_2}(t)\right)} e^{\Phi_{r_1}^{\zeta_1}(\theta_n s)},
\end{align}
so we estimate the integrals over $s$ and $t$ separately.

For $t=\mathcal{O}\left(n^{1/3}\right)$, we can estimate the logarithm
in $\operatorname{\Xi}_{w_+}^{\zeta_2}(\tilde{t})$ from above by its third 
order Taylor polynomial, giving
\begin{align}\label{globalintegral2}
\operatorname{Re}\left(\Xi_{w_+}^{\zeta_2}(\tilde{t})\right)
=&\frac{n}{2}\log\left(1+n^{-2/3}(\epsilon_n ^2+t^2)
+2\epsilon_n  n^{-1/3}\right) \nonumber\\
& -\epsilon_n  n^{2/3}
-\frac{n^{1/3}}{2\tau_n}(t^2-\epsilon_n ^2)
- \epsilon_n  \xi_2-\sigma_n t \eta_2\nonumber \\
\leq&
-\frac{1}{2}\left(\frac{\sigma_n t}{\sqrt{\tau_n}}
+\sqrt{\tau_n}\eta_2\right)^2+\frac{1}{2}\tau_n \eta_2^2
-\epsilon_n t^2+\frac{\sigma_n^2\epsilon_n^2}{2\tau_n}
-\epsilon_n \xi_2\nonumber\\
&-Q_n(t),
\end{align}
where 
\begin{equation*}Q_n(t)=-\frac{n^{-1/3}t^4}{4}\left(1-\frac{2t^2}{3n^{2/3}}\left(1+\frac{\epsilon_n^2}{t^2}+\frac{2\epsilon_n n^{1/3}}{t^2}\right)^3\right).
\end{equation*} 
To control $Q_n(t)$, we observe that if $\epsilon_n n^{1/3}T_n^{-2}\geq 1/4$, 
we can write  
\begin{equation*}
Q_n(t)\leq -\frac{n^{-1/3}t^4}{4}\left(1-C_1\frac{t^2}{n^{2/3}}
\frac{\epsilon_n^3 n}{t^6}\right)=-\frac{n^{-1/3}t^4}{4}+C_2\epsilon_n^3,
\end{equation*} 
which is either negative or of order $o(1)$. On the other hand, 
if $\epsilon_n n^{1/3}T_n^{-2}> 1/4$ we have the estimate  
\begin{equation*}
Q_n(t)\leq-\frac{n^{-1/3}t^4}{4}\left(1-\frac{2t^2}{3n^{2/3}}(1+1/3)\right)
\leq -\frac{n^{-1/3}t^4}{4}\left(1-\frac{8n^{2/3}}{9n^{2/3}}\right) < 0. 
\end{equation*}
In both cases it follows from (\ref{globalintegral2}) that
\begin{equation*}
\operatorname{Re}\left(\Xi_{w_+}^{\zeta_2}(\tilde{t})\right)\leq-\epsilon_n t^2+C_3
\end{equation*}
for $|t|\leq n^{1/3}$, so
\begin{equation}\label{boundonxi}
 \int_{T_n\leq |t|\leq n^{1/3}}e^{\operatorname{Re}\left(\Xi_{w_+}^{\zeta_2}(\tilde{t})\right)}\ud t\leq C\frac{e^{-\epsilon_n T_n^2}}{\epsilon_n T_n}.
\end{equation}

We turn to an estimate for  $|t|>n^{1/3}$. Changing variables $u=n^{-1/3}t$ and letting $\delta=(e-2)/(2e)>0$ gives an estimate
\begin{align}\label{boundonxitails}
\int_{|t|>n^{1/3}}&e^{\operatorname{Re}\left(\Xi_{w_+}^{\zeta_2}(\tilde{t})\right)}\ud t\nonumber\\
 =&n^{1/3}\int_{|u|>1}\exp\left\{\frac{n}{2}\left(\log\left(1+u^2+n^{-2/3}\epsilon_n^2+2\epsilon_n n^{-1/3}\right)\right)\right.\nonumber\\
&\left.-\frac{nu^2}{2\tau_n}-\sigma_n n^{1/3}u \eta_2 -\epsilon_n n^{2/3}+\frac{n^{1/3}}{2\tau_n}\epsilon_n^2- \epsilon_n \xi_2 \right\}\ud u\nonumber \\
 \leq&2n^{1/3}\int_{u> 1}\exp\left\{\frac{n}{2}\left(\log(1+u^2)-(1-\delta)u^2+\frac{\tau_n\sigma_n^2 n^{-4/3}\eta_2^2}{\delta}\right)\right.\nonumber\\
&\left.-\frac{n^{2/3}\epsilon_n}{2}\left(1-\epsilon_n n^{-1/3}\left(\frac{2+\tau_n}{2\tau_n}\right)\right)- \epsilon_n \xi_2 \right\}\ud u\nonumber\\
 \leq&C_1n^{1/3}\int_{u> 1}\exp\left\{\frac{n}{2}\left(\log(2u^2)-(1-\delta)u^2\right)\right\}\ud u\nonumber\\
 \leq&C_2n^{1/3}2^{n/2-1}\Gamma\left(\frac{n+1}{2}\right)
\left(\frac{2}{n(1-\delta)}\right)^{\frac{n+1}{2}}\nonumber\\
 \leq&C_3n^{-1/6}\left(\frac{6+e}{4+2e}\right)^{\frac{n+1}{2}}\nonumber\\
 \leq&Cn^{-1/6}e^{-kn},
\end{align}
for some $k>0$.
Similarly, noting that 
$\Phi_{r_1}^{\zeta_1}(\theta_n s)\leq -\epsilon_n s^2+C_1\left(n^{-1/3}s^4+1\right)$ and recalling that $\epsilon_n >n^{-2/15}$, 
\begin{align}\label{boundonphi}
\int_{T_n<|s|<\frac{\pi}{\theta_n}}e^{\Phi_{r_1}^{\zeta_1}(\theta_n s)}\ud s 
&\leq C_1\int_{T_n<|s|<\frac{\pi}{\theta_n}}
e^{\Phi_{r_1}^{\zeta_1}(\theta_n s)}\ud s\nonumber \\
 \leq C_1&\left(\frac{2\pi}{\theta_n} 
e^{\Phi_{r_1}^{\zeta_1}(\theta_n n^{1/{12}})}
+\int_{T_n<|s|<n^{1/{12}}}
e^{\Phi_{r_1}^{\zeta_1}(\theta_n s)}\ud s\right)\nonumber \\
 \leq C_2&\left(\frac{e^{-\epsilon_n n^{1/6}}}{\theta_n}+\int_{T_n<|s|<n^{1/{12}}}
e^{-\epsilon_n s^2}\ud s\right) \nonumber\\
 \leq C&
\left(n^{1/3}e^{-n^{1/30}}
+\frac{e^{-\epsilon_n T_n^2}}{\epsilon_n T_n}\right).
\end{align}
It is also clear from (\ref{boundonxi}) and (\ref{boundonxitails})  
that 
\begin{equation*}
\int_{\mathbb{R}}
e^{\operatorname{Re}\left(\Xi_{w_+}^{\zeta_2}(\tilde{t})\right)}\ud t
\asymp\int_{\mathbb{R}}e^{-\epsilon_nt^2}\ud t\leq C\epsilon_n^{-1/2}
\end{equation*}
and similarly for the $s$-integral,
 so by (\ref{factorizedF}),
\begin{multline*}\label{tairy}
\left|\iint_{((\theta_n^{-1}\mathbb{T})\times\mathbb{R})\setminus (-T_n,T_n)}
F_{r_1,w_+}^{\zeta_1,\zeta_2}(s,t-i\epsilon_n)\ud t\ud s\right| \\
\leq Cn^{-1/6}\left(\epsilon_n^{-3/2} n^{1/3}e^{-n^{1/30}}+\epsilon_n^{-3/2}
e^{-kn}
+\frac{e^{-\epsilon_n T_n^2}}{\epsilon_n^{5/2} T_n}\right)
=n^{-1/6}o(1),
\end{multline*}
which concludes the proof.

\end{proof}
Finally, we show that the contribution to the limiting integral from outside
$(-T_n,T_n)^2$ is small. 
\begin{lemma}\label{globalairy2}
Let $\epsilon_n$ and $T_n$ be as in Lemma (\ref{localairy}). 
Then, for any fixed $(\zeta_1,\zeta_2)\in \mathbb{R}^4$,   
\begin{equation*}
\left|\int_{\gamma\setminus \tilde{\gamma}_n}
\int_{\gamma\setminus \tilde{\gamma}_n}\frac{e^{-\frac{1}{2}(\sigma_n v-\eta_2)^2
+\frac{i}{3}v^3+i\xi_2v -\frac{1}{2}(\sigma_n u+\eta_1)^2+\frac{i}{3}u^3
+i\xi_1u}}{i(u+v)}\ud u \ud v\right|
=o(1)
\end{equation*}
\end{lemma}

\begin{proof}For any $u=s+i\epsilon$, 
we can estimate
\begin{align*}
&\left|\int_{\gamma\setminus \tilde{\gamma}_n}
\frac{\exp\left\{-\frac{1}{2}(\sigma v-\eta_2)^2
+\frac{i}{3}v^3+i\xi_2v\right\}}{i(u+v)} \ud v\right|\nonumber\\
 =&\left|\int_{|t|\geq T_n}\frac{\exp\left\{-\frac{1}{2}(\sigma (t+i\epsilon_n)
-\eta_2)^2+\frac{i}{3}(t+i\epsilon_n)^3+i\xi_2(t+i\epsilon_n)\right\}}
{i(s+i\epsilon_n+t+i\epsilon_n)}\ud t\right|\nonumber\\
 \leq&\int_{|t|\geq T_n}\frac{\exp\left\{-\frac{1}{2}(\sigma t-\eta_2)^2
-\epsilon_n t^2+\frac{\epsilon_n^3}{3}+\frac{\sigma^2\epsilon_n^2}{2}
-\xi_2 \epsilon_n\right\}}{2\epsilon_n}\ud t\nonumber\\
 \leq&C_1\frac{1}{\epsilon_n}\int_{|t|\geq T_n}e^{-\epsilon_nt^2}\ud t
 \leq C\frac{e^{-\epsilon_n T_n^2}}{\epsilon_n^2 T_n}.
\end{align*}
Therefore 
\begin{equation*}
\left|\int_{\gamma\setminus \tilde{\gamma}_n}\int_{\gamma}\frac{e^{-\frac{1}{2}(\sigma v-\eta_2)^2+\frac{i}{3}v^3+i\xi_2v -\frac{1}{2}(\sigma u+\eta_1)^2+\frac{i}{3}u^3+i\xi_1u}}{i(u+v)}\ud u \ud v\right| 
\leq C\frac{e^{-\epsilon_n T_n^2}}{\epsilon_n^{5/2} T_n},
\end{equation*}
and by the symmetry of the variables $u$ and $v$ the conclusion follows.

\end{proof}

\subsection{Proof of Theorem \ref{main}, part (ii)}\label{proofairy}
Using the estimates of Section \ref{estimatesairy} to prove the 
point-wise convergence of (a kernel equivalent to) $\tilde{K}_n^{\tau_n}$ to $M_{\sigma}$, and then 
proving that $\tilde{K}_n^{\tau_n}(\zeta,\zeta)$ is dominated
by an integrable function, we can apply Lemma 
\ref{convergenceoflastparticle} in this case as well.
\begin{proof}[Proof of Theorem \ref{main} (ii)]
Suppose $\sigma_n\to \sigma\in [0,\infty)$ as $n\to \infty$. 
Let $\delta_n=0$, $\alpha_n=(2\tau_n)^{-1/2}n^{1/6}$ and 
$\theta_n=n^{-1/3}$. In order to tidy up the calculations slightly, we 
will actually prove the theorem for 
 the choice of parameters $\tilde{a}=\tau_n n^{-2/3}$ and 
$\tilde{b}=\tau_n n^{-2/3}\sigma_n$, 
differing by a factor $\tau_n$ from the choice in the statement of the theorem; this 
clearly makes no difference in the limit. Thus 
$a_n=\sqrt{2\tau_n}n^{-1/6}$  
and $b_n=\sqrt{2\tau_n(1-\tau_n)}$ as defined by (\ref{defparameters}).  
Given $\epsilon_n$ and $T_n$ as in Lemmas 
\ref{localairy} and \ref{globalairy}, 
choose $r_1=w_--\tau_n\alpha_n\epsilon_n$ and 
$r_2=w_++\alpha_n\epsilon_n$. Let 
$(\zeta_1,\zeta_2)\in \mathbb{R}^4$ be fixed.


By the estimates of Lemmas \ref{localairy}, 
\ref{globalairy} and \ref{globalairy2}, it follows 
that 
\begin{multline}
\frac{n^{1/6}}{(2\pi)^{5/2}}
\exp\left\{-\frac{1}{2}(\eta_1^2+\eta_2^2)+f_n(r_2)-f_n(w_+)
-\alpha_n \epsilon_n(a_n \xi_2-i b_n\eta_2)\right\}\\
\times\iint_{(-\pi/\theta,\pi/\theta)\times\mathbb{R}}
F_{r_1,r_2}^{\zeta_1,\zeta_2}(s,t)\ud t\ud s\\
=-M_{\sigma_n}(\zeta_1,\zeta_2)+ o(1).
\end{multline}
Inserting into (\ref{rescaledkernel}) and comparing with the kernel 
$M_{\sigma_n}$, results in the expression
\begin{align*}
\tilde{K}_n^{\tau_n}&(\zeta_1,\zeta_2)/{M_{\sigma_n}(\zeta_1,\zeta_2)}\nonumber\\ 
 =&
\exp \left\{ i\tau_n \sqrt{n(1-\tau_n)}(\eta_1-\eta_2)
+\frac{\tau_n n^{-1/3}}{2(1+\tau_n)}(\xi_2^2-\tau_n\xi_1^2)\right.
\nonumber\\
&\left. +\frac{1}{2}((1-\tau_n^2)\eta_1^2+(1-\tau_n)\eta_2^2)
+\frac{1}{2}\epsilon_n^2\sigma_n^2
 \right\}(1+o(1))\nonumber\\
 =&\exp \left\{ i\tau_n \sqrt{n(1-\tau_n)}(\eta_1-\eta_2)\right\}(1+o(1)).
\end{align*}
Considering the equivalent kernel 
\begin{equation*}
{K_n^{\tau_n}}'(\zeta_1,\zeta_2)
=\exp\{-i\tau_n \sqrt{n(1-\tau_n)}(\eta_1-\eta_2)\}\tilde{K}_n^{\tau_n}(\zeta_1,\zeta_2),
\end{equation*}
we see that
\begin{equation*}
{K_n^{\tau_n}}'(\zeta_1,\zeta_2)=M_{\sigma_n}(\zeta_1,\zeta_2)(1+o(1))\to 
M_{\sigma}(\zeta_1,\zeta_2) \textrm{ as } n\to \infty.
\end{equation*}

It remains to prove that $\tilde{K}_n^{\tau_n}(\zeta,\zeta)$ is dominated, 
for every sufficiently large $n$, by an integrable function on $(\xi_0,\infty)\times\mathbb{R}$.
This part of the proof is similar to the corresponding step in the proof 
of the first part of Theorem \ref{main} in Section \ref{proofgumbel}, 
although
the subtleties of that case are not present here.

Let $\epsilon<1$ be fixed and consider any $n$ so large that 
\begin{equation}\label{ncondition}
|\xi_0|<\epsilon n^{1/3}.
\end{equation}
Choose 
\begin{equation*}
r_1=\sqrt{\frac{\tau_n n}{2}}(1-\epsilon_1)
\end{equation*}
and
\begin{equation*}
r_2=\sqrt{\frac{n}{2\tau_n}}(1+\epsilon_1),
\end{equation*}
where $\epsilon_1=(1+\sigma_n^2)n^{-1/3}$. 
To estimate the integral 
\begin{equation*}
\iint_{\left(\theta_n^{-1}\mathbb{T}\right)
\times\mathbb{R}}F_{r_1,r_2}^{\zeta,\zeta}(s,t)\ud t\ud s,
\end{equation*}
replace the real line by the contour  
$\mathbb{R}+ia_n(1-\tau_n)\xi/{(4\alpha_n)}$ 
when integrating in $t$,
which is allowed by the condition (\ref{ncondition}) on $n$, and
put 
\begin{equation*}
t'=t+i\delta \xi := t+i\frac{a_n(1-\tau_n)\xi}{4\alpha_n}.
\end{equation*} 
This gives
\begin{multline}\label{dominatedairy1}\left|\iint_{\left(\theta_n^{-1}\mathbb{T}\right)
\times\mathbb{R}}F_{r_1,r_2}^{\zeta,\zeta}(s,t)\ud t\ud s\right|
\leq \iint_{\left(\theta_n^{-1}\mathbb{T}\right)
\times\mathbb{R}}
\left|F_{r_1,r_2}^{\zeta,\zeta}(s,t'-i\delta\xi)\right|\ud t'\ud s\\
\leq \frac{1}{\left(\sqrt{2\tau_n n}\epsilon_1
-\tau_n\alpha_n \delta |\xi_0|\right)} \int_{\theta_n^{-1}\mathbb{T}}
e^{\Phi_{r_1}^{\zeta}(\theta_n s)}\ud s\int_{\mathbb{R}}
e^{\operatorname{Re}\left(\Xi_{r_2}^{\zeta}(t'-i\delta\xi)\right)}\ud t'.
\end{multline}

Using the inequality $\log(1+x)\leq x$ provides an estimate
\begin{align}\label{dominatedairy2}
\int_{\mathbb{R}}&e^{\operatorname{Re}\left(\Xi_{r_2}^{\zeta}(t'-i\delta\xi)\right)}\ud t'\nonumber\\
 \leq&\int_{\mathbb{R}}
\exp\left\{-t'^2\left(\frac{3\sigma_n^2+2-2(1+\sigma_n^2)\epsilon/r_2}{2\tau_n(1+\epsilon_1)}\right)-\sigma_n\eta t' \right.\nonumber\\
&\left.
-\alpha_n \delta (a_n-\alpha_n\delta)\xi^2
+\alpha_n\delta\left(\frac{n}{r_2}-c_n+2r_2\right)\xi\right\}\ud t'\nonumber\\
 \leq&
C\exp\left\{\frac{\eta^2}{6}(1+\epsilon) 
-\alpha_n \delta(a_n-\alpha_n\delta)\xi^2
+\alpha_n\delta(n/r_2-c_n+2r_2)\xi\right\}.
\end{align}
Suppose without loss of generality that $\eta \leq 0$. 
Then there is a 
$v_{\eta}\in [0,\pi/2)$ such that  ${\Phi_{w_-}^{\zeta}}'(v_{\eta})=0$ 
and since 
\begin{align*}
{\Phi_{r_1}^{\zeta}}''(v_{\eta})
 =&4r_1^2\cos 2v_{\eta}-r_1c_n' \cos v_{\eta}+r_1b_n\eta\sin v_{\eta}\nonumber\\
<&4r_1^2-r_1c_n'\nonumber\\
<&-n^{2/3}\sigma_n^2-\epsilon_1n(3\tau_n-1)
+\sqrt{n}a_n|\xi_0|\nonumber\\
<&-n^{2/3}(1+3\tau_n\sigma_n^2)<0,
\end{align*}
a saddle point argument shows that
\begin{multline*}
\int_{|\theta_n s|<\pi}e^{\Phi_{r_1}^{\zeta}(\theta_n s)}\ud s\\
 \leq C_1e^{\Phi_{w_-}^{\zeta}(v_{\eta})}\int_{\mathbb{R}}
\exp\left\{-\frac{(r_1c_n'-4r_1^2)\theta_n^2}{2}(s-v_{\eta}/\theta_n)^2\right\} \ud s
 \leq Ce^{\Phi_{w_-}^{\zeta}(v_{\eta})}.
\end{multline*}
Consider the difference
\begin{multline*}
\frac{\eta^2}
{6}-\Phi_{r_1}^{\zeta}(v_{\eta})\\
=\frac{1}
{6}\left(\eta^2+6r_1b_n \eta\sin v_{\eta}
-6r_1c_n'\left((\cos v_{\eta}-1)+r_1^2(1-\cos2 v_{\eta})\right)\right)\\
\geq -\frac{3}{2}b_n^2r_1^2\sin^2 (v_{\eta})-r_1 c_n'(\cos v_{\eta}-1)-r_1^2(1-\cos2 v_{\eta})=:g(v_{\eta}).
\end{multline*}
Since 
\begin{align*}
 g'(v_{\eta}) =&
r_1^2(4-3b_n^2)\sin v_{\eta}\left(\frac{c_n'}{r_1(4-3b_n^2)}-\cos v_{\eta}\right)\nonumber\\
\geq& r_1^2(4-3b_n^2)\sin v_{\eta}\left(\frac{c_n-a_n|\xi_0|}{4r_1}-\cos v_{\eta}\right)\geq0,
\end{align*}
$g$ is increasing and positive on $[0,\pi/2)$, which shows that
\begin{equation}\label{dominatedairy3}
\int_{|\theta_n s|<\pi}e^{\Phi_{r_1}^{\zeta}(\theta_n s)}\ud s
\leq  Ce^{\eta^2/6}.
\end{equation}
Recall that $r_2$ is ``almost'' a saddle point of $f_n$, so 
\begin{equation*}\left|\frac{n}{r_2}-c_n+2r_2\right|=\left|\frac{\sqrt{2n}\epsilon_1((1-\tau_n)+\epsilon_1)}{\sqrt{\tau_n}(1+\epsilon_1)}\right|\leq 6\sqrt{2}(1+\sigma_n^4)n^{-1/6}.
\end{equation*} 
Therefore, by (\ref{rescaledkernel}) and the estimates (\ref{dominatedairy1}), (\ref{dominatedairy2}) and (\ref{dominatedairy3}),
\begin{align*}
\tilde{K}_n^{\tau_n}(\zeta_1,\zeta_2)
 \leq&
C_{\sigma}\exp\left\{a_n^2(1-\tau_n)\left(\frac{1}{4(1+\tau_n)}
+\frac{(1-\tau_n)}{16}-\frac{1}{4}\right) \xi^2\right\}\nonumber\\
&\times\exp\left\{-\left((1+\tau_n)(1+\sigma_n^2)
+\alpha_n\delta(n/r_2-c_n+2r_2)\right)\xi\right\}\nonumber\\
&\times\exp\left\{-\left(\frac{\tau_n(1+\tau_n)}{2}
-\frac{2+\epsilon}{6}\right)\eta^2\right\}\nonumber\\
 \leq& C_{\sigma, \xi_0}e^{-\xi-\eta^2/2},
\end{align*}
which is an integrable function on $(\xi_0,\infty)\times \mathbb{R}$. 
It follows from Lemma (\ref{convergenceoflastparticle}) and 
the point-wise convergence 
of ${K_n^{\tau_n}}'$ to $M_{\sigma}$ that  $\tilde{Z}_n^{\tau_n}$ 
converges weakly to $Z_{\sigma}$ 
and that the last particle distribution $F_n^{\tau_n}(t)$ converges to $F_{\sigma}(t)$.
\end{proof}


\emph{Acknowledgement:} I would like to express my sincere 
gratitude to Kurt Johansson for all his support and advice.
\newpage
\appendix
\section{}\label{Preliminaries}

 For sequences of determinantal processes, the following criterion relates 
weak convergence to convergence of the correlation kernels.
\begin{lemma}\label{convergencecriterion}
For each $n\geq 1$, let $X_n$ be a determinantal point processes on a complete 
separable metric space $\Lambda$, with correlation kernel $K_n$. If 
$K:\Lambda^2 \to \mathbb{C}$ is a function such that for every 
compact set $A\subseteq \Lambda$, 
\begin{equation}
\label{convergencecrit}
\sum_{k=0}^{\infty}\frac{1}{k!}\int_{A^k}
\left|\det\left(K_n(\xi_i,\xi_j)\right)_{i,j=1}^k
- \det\left(K(\xi_i,\xi_j)\right)_{i,j=1}^k\right|
\ud^k\lambda(\xi)  \to 0 \textrm{ as } n\to \infty,
\end{equation}
then $K$ is the correlation kernel of a determinantal 
point process $X$, and $X_n$ converges weakly to $X$ as $n\to \infty$.
\end{lemma}
\begin{proof} Let $\{A_i\}_{j=1}^k$ be a family of bounded disjoint Borel sets in 
$\Lambda$ and $t=(t_1,\ldots,t_k)\in [-\ln 2,\infty)^k$ be given. 
Define the function 
$\phi_t(s) =\sum_{j=1}^k(e^{-t_i}-1)\chi_{A_j}(s)$. The support of $\phi$ is the 
compact set $A=\overline{\cup_{j=1}^k A_j}$, and $|\phi(s)|\leq 1$.
The Laplace transform of the $k$-dimensional distributions of $X_n$ can then 
be written
\begin{align}
\Psi_n(t):
=&\mathbb{E}_n\left[\exp\left\{-\sum_{j=1}^kt_j|X_n\cap A_j|\right\}\right]\nonumber\\
 =&\mathbb{E}_n\left[\prod_{m}\left(1+\phi_t(x_m)\right)\right]\nonumber\\
 =&\sum_{r=0}^{\infty}\frac{1}{r!}\int_{\Lambda^r}\prod_{j=1}^r\phi_t(\xi_j)\rho_r^n(\xi_1,\ldots,\xi_r)\ud^r\lambda(\xi)\nonumber\\
 =&\sum_{r=0}^{\infty}\frac{1}{r!}\int_{\Lambda^r}\prod_{j=1}^r\phi_t(\xi_j)\det\left(K_n(\xi_i,\xi_j)\right)_{i,j=1}^r \ud^r\lambda(\xi)\nonumber\\
&\to\sum_{r=0}^{\infty}\frac{1}{r!}\int_{\Lambda^r}\prod_{j=1}^r\phi_t(\xi_j)\det\left(K(\xi_i,\xi_j)\right)_{i,j=1}^r \ud^r\lambda(\xi)\nonumber
\end{align}
as $n \to \infty$, by (\ref{defcorrelation2}) and the hypothesis.
Convergence of the Laplace transforms for all $t$ in a neighbourhood 
of the origin 
implies convergence of the joint distributions. The existence of 
a determinantal 
point process with correlation kernel $K$ is an immediate consequence 
of the fact 
that the $K_n$ all are correlation kernels; the consistency conditions 
 e.g. in  \cite{Soshnikov} are easily seen to be satisfied. 

\end{proof}

\begin{lemma}\label{lastparticlelemma} 
Let $Z$ be a determinantal point process on  $\mathbb{R}^2$ with a Hermitian 
correlation kernel $K$ and suppose that
\begin{equation*}
\int_{(\xi_0,\infty)\times \mathbb{R}}K(\zeta,\zeta)\ud \zeta<\infty
\end{equation*}
for every $\xi_0\in \mathbb{R}$. Then $Z$ has a last particle almost surely and the 
distribution function $F$ of the last particle 
 is given by
\begin{equation}\label{lastparticle2}
F(t)=\sum_{r=0}^{\infty}\frac{(-1)^r}{r!}\int_{\left((t,\infty)\times \mathbb{R}\right)^r}\det (K(\zeta_i,\zeta_j))_{i,j=1}^{r}\ud^r \zeta.
\end{equation}
If $K$ defines a locally trace class integral operator on $L^2(\mathbb{R}^2)$, $F(t)$ can be more compactly expressed as a Fredholm determinant
\begin{equation}F(t)=\det(I-K)_{L^2\left((t,\infty)\times\mathbb{R}\right)}.
\end{equation}
\end{lemma}

\begin{proof}
Let $t\in \mathbb{R}$ be given and suppose $s>t$. 
Applying (\ref{defcorrelation2}), with 
$\phi(s,t)=\chi_{\left((t,s)\times(-s,s)\right)}$, gives
\begin{multline*}
\mathbb{P}\left[|Z\cap\left((t,s)\times (-s,s)\right)|=0\right]\\
=\sum_{r=0}^{\infty}\frac{(-1)^r}{r!}\int_{\left((t,s)\times (-s,s)\right)^r}
\det (K(\zeta_i,\zeta_j))_{i,j=1}^{r}\ud^r\zeta. 
\end{multline*}
Because $K$ is Hermitian,  the matrix 
$ (K(\zeta_i,\zeta_j))_{i,j=1}^{r}$ is positive definite, 
so by Hadamard's inequality 
\begin{align*}
 \sum_{r=0}^{\infty}\frac{1}{r!}\int_{\left((t,\infty)\times \mathbb{R}\right)^r}
\det (K(\zeta_i,\zeta_j))_{i,j=1}^{r}\ud^r\zeta
 \leq&  \sum_{r=0}^{\infty}\frac{1}{r!}\left(\int_{(t,\infty)\times \mathbb{R}}
K(\zeta,\zeta)\ud\zeta\right)^r\nonumber\\
 =&\exp\left\{\int_{(t,\infty)\times \mathbb{R}}
K(\zeta,\zeta)\ud\zeta\right\}<\infty.
\end{align*}
Since 
correlation functions are non-negative, 
\begin{multline*}
\left|\frac{(-1)^r}{r!}\int_{\left((t,s)\times (-s,s)\right)^r}
\det (K(\zeta_i,\zeta_j))_{i,j=1}^{r}\ud^r\zeta\right|\\
\leq
\frac{1}{r!}\int_{\left((t,\infty)\times \mathbb{R}\right)^r}
\det (K(\zeta_i,\zeta_j))_{i,j=1}^{r}\ud^r\zeta,
\end{multline*}
so by the dominated convergence theorem
\begin{align*}
F(t) =&\lim_{s\to \infty}\mathbb{P}
\left[|Z\cap\left((t,s)\times (-s,s)\right)|=0\right]\nonumber\\
 =&\sum_{r=0}^{\infty}\frac{(-1)^r}{r!}\int_{\left((t,\infty)\times \mathbb{R}\right)^r}
\det (K(\zeta_i,\zeta_j))_{i,j=1}^{r}\ud^r\zeta.
\end{align*}
\end{proof}

\begin{proof}[Proof of Lemma \ref{convergenceoflastparticle}]
We begin with the convergence of the last particle distribution. 
The existence and distribution of a last particle for $Z_n$ is given 
by Lemma  \ref{lastparticlelemma}.
By Hadamard's inequality,
\begin{equation*}
\left|\frac{(-1)^r}{r!}
\det \left(K_n(\zeta_i,\zeta_j)\right)_{i,j=1}^r\right|
\leq \frac{1}{r!}
\prod_{i=1}^rK_n(\zeta_i,\zeta_i)\leq \frac{1}{r!}\prod_{i=1}^rB(\zeta_i).
\end{equation*}
Now
\begin{equation*}
\int_{\left((t,\infty)\times \mathbb{R}\right)^r} 
\frac{1}{r!}\prod_{i=1}^rB(\zeta_i)\ud^r\zeta
=\frac{1}{r!}
\left(\int_{(t,\infty)\times \mathbb{R}}B(\zeta)\ud\zeta\right)^r
=\frac{C_t^r}{r!}
\end{equation*}
and
\begin{equation*}
\sum_{r=0}^{\infty}\frac{C_t^r}{r!}=e^{C_t}<\infty,
\end{equation*}
so applying the dominated convergence theorem twice gives the conclusion
\begin{align*}
\lim_{n\to \infty}F_n(t)
=&\lim_{n\to \infty}\sum_{r=0}^{\infty}
\int_{\left((t,\infty)\times \mathbb{R}\right)^r}
\frac{(-1)^r}{r!}
\det \left(K_n(\zeta_i,\zeta_j)\right)_{i,j=1}^r\ud^r\zeta\nonumber\\
 =&\sum_{r=0}^{\infty}
\int_{\left((t,\infty)\times \mathbb{R}\right)^r}\lim_{n\to \infty}
\frac{(-1)^r}{r!}
\det \left(K_n(\zeta_i,\zeta_j)\right)_{i,j=1}^r\ud^r\zeta\nonumber\\
 =&\sum_{r=0}^{\infty}
\int_{\left((t,\infty)\times \mathbb{R}\right)^r}
\frac{(-1)^r}{r!}
\det \left(K(\zeta_i,\zeta_j)\right)_{i,j=1}^r\ud^r\zeta\nonumber\\
 =&F(t).
\end{align*}
To verify the condition for weak convergence of point processes 
stated in Lemma \ref{convergencecriterion}, the same argument 
applies since an arbitrary compact set $A\subset \mathbb{R}^2$ 
is contained in a set of the form $(t,\infty)\times \mathbb{R}$.
\end{proof}

\end{document}